\documentclass[aop]{imsart}

\usepackage{epsfig,amsfonts,amssymb, url, amsthm, subcaption, hyperref, tikz,
amsmath, calrsfs, bigdelim, multirow, appendix, float, mathtools,
stmaryrd, environ, enumitem, dsfont}

\usepackage[normalem]{ulem}
\usepackage[utf8]{inputenc} 
\usepackage[T1]{fontenc}

\startlocaldefs
\newtheorem{theorem}{Theorem}[section]
\newtheorem{lemma}[theorem]{Lemma}
\newtheorem{proposition}[theorem]{Proposition}
\newtheorem{corollary}[theorem]{Corollary}
\newtheorem{definition}[theorem]{Definition}
\newtheorem{assumption}[theorem]{Assumption}
\newtheorem{remark}[theorem]{Remark}
\newtheorem{notation}[theorem]{Notation}
\theoremstyle{definition}

\newcommand{\RR}{\mathbb{R}}
\newcommand{\NN}{\mathbb{N}}
\newcommand{\ZZ}{\mathbb{Z}}

\newcommand{\EE}{\mathbb{E}}

\newcommand{\PP}{\mathbb{P}}
\newcommand{\TT}{\mathbb{T}}
\newcommand{\DD}{\mathbb{D}}

\newcommand{\mL}{\mathcal{L}}

\newcommand{\mC}{\mathcal{C}}

\newcommand{\mM}{\mathcal{M}}

\newcommand{\mI}{\mathcal{I}}

\newcommand{\mF}{\mathcal{F}}
\newcommand{\mD}{\mathcal{D}}

\newcommand{\mE}{\mathcal{E}}
\newcommand{\mS}{\mathcal{S}}
\newcommand{\mB}{\mathcal{B}}

\newcommand{\mH}{\mathcal{H}}

\newcommand{\mf}[1]{\mathfrak{#1}}
\newcommand{\mfd}{\mathfrak{d}}

\newcommand{\mfL}{\mathfrak{L}}

\newcommand{\para}{\varolessthan}

\newcommand{\reso}{\varodot}

\newcommand{\hh}{\frac{1}{2}}
\newcommand{\ve}{\varepsilon}
\newcommand{\vr}{\varrho}
\newcommand{\vt}{\vartheta}

\newcommand{\lqm}{``}
\newcommand{\rqm}{'' }
\newcommand*{\ud}{\mathrm{\,d}}

\newcommand{\RD}{\mathbb{R}^d}

\newcommand{\weights}{\boldsymbol{\vr}(\omega)}

\newcommand{\znd}{\mathbb{Z}^d_n}

\newcommand{\supp}{\operatorname{supp}}

\newcommand{\fn}{\lfloor n^{\vr} \rfloor}
\endlocaldefs

\begin{document}

\begin{frontmatter}

\title{A rough super-Brownian motion}
\runtitle{Rough super-Bm}

\begin{aug}
\author[A]{\fnms{Nicolas} \snm{Perkowski}\ead[label=e1]{perkowski@math.fu-berlin.de}},
\and
\author[A]{\fnms{Tommaso} \snm{Rosati}\ead[label=e2]{rosati.tom@fu-berlin.de}}
\address[A]{Freie Universit\"at Berlin, Institut f\"ur Mathematik, Arnimallee 7, 14195 Berlin, Germany \\ \printead{e1,e2}}
\end{aug}


\date{\today}

\begin{abstract} We study the scaling limit of a branching random walk in static 
  random environment in dimension $d=1,2$ and show that it is given by a
  super-Brownian motion in a white noise potential. In dimension $1$ we
  characterize the
  limit as the unique weak solution to the stochastic PDE: \[\partial_t \mu =
  (\Delta {+} \xi) \mu  {+} \sqrt{2\nu \mu} \tilde{\xi}\] for independent
  space white noise $\xi$ and space-time white noise $\tilde{\xi}$. In
  dimension $2$ the study requires paracontrolled theory and the limit process
  is described via a martingale problem. In both dimensions we prove persistence
of this rough version of the super-Brownian motion.  \end{abstract}



\end{frontmatter}

\maketitle 

\setcounter{tocdepth}{1} 
\tableofcontents

\section*{Introduction}

This work explores the large scale behavior of a branching random walk in a
random environment (BRWRE). Such process is a particular kind of spatial
branching process on $\ZZ^d$, in which the branching and killing rate of a
particle depends on the value of a potential $V$ in the position of the
particle.  In the model analyzed in this work, the dimension is restricted to
$d=1,2$ and the potential is chosen at random on the lattice: \[V(x) = \xi(x), \
  \text{ with } \ \{ \xi(x) \}_{x \in \ZZ^d} \ \text{ i.i.d., } \ \xi(x) \sim
\Phi\] for a given random variable $\Phi$ (normalized via $\EE \Phi =
0, \EE \Phi^2 = 1$).

A particle $X$ in this process at time \(t\) jumps to a nearest neighbor at rate
\(1\), gives birth to a particle at rate \(\xi(X(t))_+\) or dies at rate
\(\xi(X(t))_{-}\).  After branching, the new and the old particle follow the
same rule independently of each other.

The BRWRE is used as a model for chemical reactions or biological processes,
e.g. mutation, in a random medium. This model is especially interesting in
relation to intermittency and localization
\cite{ZeldocivhMolchanovRuzmauikinSokolov1987, GaertnerMolchanov1990,
AlbeverioBogachevEtAl2000, GunKonigSekulovic2013}, and other large times
properties such as survival \cite{BartschGantertKochler2009,
GantertMullerPopov2010}. 

Scaling limits of branching particle systems have been an active field of
research since the early results by Dawson et al. and gave rise to the study of
superprocesses, most prominently the so-called super-Brownian
motion (see \cite{Etheridge2000, DawsonMaisonneuve1993SaintFlour} for excellent
introductions). This work follows the original setting and studies the behavior
under diffusive scaling: Spatial increments $\Delta x \simeq 1/n$, temporal
increments $\Delta t \simeq 1/n^2$. The particular nature of our problem
requires us to couple the diffusive scaling with the scaling of the environment:
This is done via an \lqm averaging parameter\rqm $\vr \geq d/2$, while the noise
is assumed to scale to space white noise (i.e. $\xi^n(x) \simeq n^{d/2}$).

The diffusive scaling of spatial branching processes in a random environment has
already been studied, for example by Mytnik \cite{Mytnik1996}.  As opposed to
the current setting, the environment in Mytnik's work is white also in time.
This has the advantage that the model is amenable to probabilistic martingale
arguments, which are not available in the static noise case that we
investigate here. Therefore, we replace some of the probabilistic tools with
arguments of a more analytic flavor. Nonetheless, at a purely formal level our
limiting process is very similar to the one obtained by Mytnik: See for example
the SPDE representation~\eqref{eqn:SPDE_rsbm} below. Moreover, our approach is
reminiscent of the conditional duality appearing in later works by Cri\c{s}an
\cite{Crisan2004Superbrown}, Mytnik and Xiong
\cite{MytnikXiong2007LocalExtinction}. Notwithstanding these resemblances, we
shall see later that some statistical properties of the two processes differ
substantially.

At the heart of our study of the BRWRE lies the following observation. If
$u(t,x)$ indicates the numbers of particles in position $x$ at time $t$, then
the conditional expectation given the realization of the random environment,
$w(t,x) = \EE[u(t,x)|\xi]$, solves a linear PDE with stochastic coefficients
(SPDE), which is a discrete version of the parabolic Anderson model (PAM):
\begin{equation}\label{eqn:pam} \partial_t w(t,x)  = \Delta w(t,x) + \xi (x)
w(t,x), \ \  (t,x) \in \RR_{> 0} \times \RR^d, \ \  w(0, x) = w_0(x).
\end{equation}

The PAM has been studied both in the discrete and in the continuous
setting (see \cite{Konig2016} for an overview). In the latter case ($\xi$ is
space white noise) the SPDE is not solvable via It\^o integration theory, which
highlights once more the difference between the current setting and the work by
Mytnik. In particular, in dimension $d = 2,3$ the study of the continuous PAM
requires special analytical and stochastic techniques in the spirit of rough
paths, such as the theory of regularity structures \cite{Hairer2014} or of
paracontrolled distributions \cite{GubinelliImkellerPerkowski2015}. In dimension
$d = 1$ classical analytical techniques are sufficient. In dimension $d \geq 4$
no solution is expected to exist, because the equation is no longer
\textit{locally subcritical} in the sense of Hairer~\cite{Hairer2014}.  The dependence of the subcriticality
condition on the dimension is explained by the fact that white noise loses
regularity as the dimension increases.

Moreover, in dimension $d = 2,3$ certain functionals of the white noise need to
be tamed  with a technique called
\textit{renormalization,} with which we remove diverging singularities.
In this work, we restrict to dimensions $d = 1,2$ as this simplifies several
calculations.  At the level of the $2$-dimensional BRWRE, the renormalization has the
effect of slightly tilting the centered potential by considering instead an
effective potential: \[\xi^n_e(x) = \xi^n(x) {-} c_n, \qquad c_n \simeq \log
(n).\]
So if we take the average over the environment, the system is slightly out of criticality, in the biological sense, namely  births are less likely than deaths. This asymmetry is counterintuitive at first. Yet, as we will discuss later, the random environment has a strongly benign effect on the process, since it generates extremely favorable regions. These favorable regions are not seen upon averaging, and they have to be compensated for by subtracting the renormalization.

The special character of the noise and the analytic tools just highlighted will allow us, in a nutshell, to fix one realization of the environment - outside a
null set -  and derive the following scaling limits. For \lqm averaging parameter\rqm $\vr >d/2$ a law of large numbers holds: The
process converges to the continuous PAM. Instead, for $\vr =d/2$ one captures
 fluctuations from the branching mechanism. The limiting process can be
characterized via duality or a martingale problem (see Theorem \ref{thm:CLT})
and we call it \emph{rough super-Brownian motion (rSBM)}. In
dimension $d=1$, following the analogous results for SBM by
\cite{KonnoShiga1988, Reimers1989}, the rSBM admits a density which in turn
solves the SPDE:
\begin{equation}\label{eqn:SPDE_rsbm} \partial_t \mu(t,x) =
  \Delta \mu(t,x) {+} \xi(x)\mu(t,x){+} \sqrt{2\nu \mu(t,x)} \tilde{\xi}(t,x), \
  \ (t, x) \in \mathbb{R}_{\geq 0} \times \RR,
\end{equation}
with \(\mu(0, x) = \delta_0(x)\), where $\tilde{\xi}$ is space-time white noise
that is independent of the space white noise $\xi$, and where $\nu = \mathbb
E\Phi^+$. The solution is weak both in the probabilistic and in the analytic
sense (see Theorem \ref{thm:rsbm_SPDE} for a precise statement). This means
that the last product represents a stochastic integral in the sense of Walsh
\cite{Walsh1986} and the space-time noise is constructed from the solution.
Moreover, the product $\xi \cdot \mu$ is defined only upon testing with
functions in the random domain of the Anderson Hamiltonian $\mH = \Delta {+}
\xi$, a random operator that was introduced by
Fukushima-Nakao~\cite{Fukushima1976} in $d=1$ and by
Allez-Chouk~\cite{Allez2015} in $d=2$, see also \cite{Gubinelli2018Semilinear, Labbe2018} for $d=3$. 

One of the main motivations for this work was the aim to understand the
SPDE~\eqref{eqn:SPDE_rsbm} in $d=1$ and the corresponding martingale problem in
$d=2$. For $\tilde \xi = 0$, equation~\eqref{eqn:SPDE_rsbm} is just the PAM
which we can only solve with pathwise methods, while for $\xi = 0$ we obtain
the classical SBM, for which the existence of pathwise solutions is a long
standing open problem and for which only probabilistic martingale
techniques exist. (See however~\cite{Chakraborty2018} for some progress on
finite-dimensional rough path differential equations with square root
nonlinearities).
Here we combine these two approaches via a 
mild formulation of the martingale problem based on the Anderson Hamiltonian.
A similar point of view was recently taken by
Corwin-Tsai~\cite{Corwin2018}, and to a certain extent
also in~\cite{Gubinelli2018Semilinear}.

Coming back to the rSBM, we conclude this work with a proof of persistence of
the process in dimension $d=1,2$. More precisely we even show that with positive
probability we have $\mu(t,U) \to \infty$ for all open sets $U \subset \RR^d$. This is opposed to what happens for the classical SBM,
where persistence holds only in dimension $d \geq 3$, whereas in dimensions
$d=1,2$ the process dies out: See \cite[Section 2.7]{Etheridge2000} and the
references therein. Even more striking is the difference between our process and
the SBM in random, white in time, environment: Under the assumption of a
heavy-tailed spatial correlation function Mytnik and Xiong
\cite{MytnikXiong2007LocalExtinction} prove extinction in finite time in any
dimension. Note also that in \cite{Etheridge2000,
MytnikXiong2007LocalExtinction} the process is started in the Lebesgue measure,
whereas here we prove persistence if the initial value is a Dirac mass.
Intuitively, this phenomenon can be explained by the presence of ``very
favorable regions'' in the random environment.

\section*{Structure of the Work}

In Assumption \ref{assu:noise} we introduce the probabilistic requirements on the
random environment. These assumptions allow us to fix a null set outside of
which certain analytical conditions are satisfied, see
Lemma~\ref{lem:renormalisation} for details. We then introduce the model,
(a rigorous construction of the random Markov process is postponed to Section
\ref{sectn:construction_markov_process} of the Appendix). We also state the main
results in Section~\ref{sectn:model}, namely the law of large numbers (Theorem
\ref{thm:LLN}), the convergence to the rSBM (Theorem \ref{thm:CLT}), the
representation as an SPDE in dimension $d=1$ (Theorem \ref{thm:rsbm_SPDE}) and
the persistence of the process (Theorem \ref{thm:persistence}). We then proceed
to the proofs. In Section \ref{sectn:pam} we study the discrete and continuous
PAM. We recall the results from \cite{MartinPerkowski2017} and adapt
them to the current setting. 

We then prove the convergence in distribution of the BRWRE in
Section~\ref{sectn:rSBM}. First, we show tightness by using a mild martingale
problem (see Remark \ref{rem:mild_martingale_problem_discrete}) which fits well
with our analytical tools. We then show the duality of the process to the SPDE
\eqref{eqn:PDE_laplace_duality_continuous} and use it to deduce the uniqueness
of the limit points of the BRWRE.

In Section~\ref{sectn:properties} we derive some properties of the rough
super-Brownian motion: We show that in $d=1$ it is the weak solution to an SPDE,
where the key point is that the random measure admits a density w.r.t. the
Lebesgue measure,
as proven in Lemma \ref{lem:existence_of_density}. We also show that the process
survives with positive probability, which we do by relating it to the rSBM on a
finite box with Dirichlet boundary conditions and by applying the spectral
theory for the Anderson Hamiltonian on that box. For this we rely on
\cite{Chouk2019} and \cite{Rosati2019kRSBM}.

\section*{Acknowledgements}  

The authors would like to thank Adrian Martini for thoroughly reading this work and pointing
out some mistakes and improvements.

The main part of the work was done while N.P. was employed at Max-Planck-Institut f\"ur Mathematik in den Naturwissenschaften Leipzig \& Humboldt-Universit\"at zu Berlin, and while T.R. was employed at Humboldt-Universit\"at zu Berlin.

N.P. gratefully acknowledges financial support by the DFG via the Heisenberg program. 

T.R. gratefully acknowledges financial support by the DFG / FAPESP: this paper was developed within the scope of the IRTG 1740 / TRP 2015/50122-0.

\section{Notations}\label{sectn:notations}

We define $\NN = \{1, 2, \ldots \}$, $\NN_0 = \NN \cup \{ 0 \}$ and \(\iota =
\sqrt{{-} 1} \).  We write $\znd$ for the lattice $\frac{1}{n} \ZZ^d$, for $n
\in \NN$, and since it is convenient we also set $\ZZ^d_\infty = \RR^d$. Let us
recall the basic constructions from \cite{MartinPerkowski2017}, where
paracontrolled distributions on lattices were developed. Define the Fourier
transforms for $k, x \in \RR^d$
\[\mF_{\RR^d}(f)(k) = \int_{\RR^d} \ud x \ f(x) e^{-2\pi \iota \langle x, k
\rangle}, \ \mF_{\RR^d}^{-1}(f)(x) = \int_{\RR^d} \ud k \ f(k) e^{2\pi \iota
\langle x, k \rangle}\] 
as well as for $x \in \znd , k\in \TT^d_n$ (with
$\TT^d_n = (\RR/(n\ZZ))^d$ the $n$-dilatation of the
torus $\TT^d$): 
\begin{align*} \mF_n (f) (k) &= \frac{1}{n^d} \sum_{x \in \znd} f(x)e^{-2\pi \iota
\langle x, k \rangle}, \qquad k \in \TT^d_n, \\
 \mF^{-1}_n (f) (x) & = \int_{\TT^d_n} \ud
k \ f(k)e^{2\pi \iota \langle x, k \rangle}, \qquad x\in \znd.  
\end{align*}

Consider $\omega(x) = |x|^{\sigma}$ for some $\sigma \in (0,1)$. We then define
$\mS_\omega$ and $\mS_{\omega}'$ as in
\cite[Definition~2.8]{MartinPerkowski2017}. Roughly speaking $\mS_\omega$ is a
subset of the usual Schwartz functions, and $\mS_{\omega}'$ consists of
so-called \emph{ultradistributions}, with more permissive growth conditions at
infinity. Let $\weights$ be the space of admissible weights as in
\cite[Definition~2.7]{MartinPerkowski2017}. For our purposes it suffices to
know that for any $a \in \RR_{\ge 0}, l \in \RR$, the functions $p(a)$ and
$e(l)$ belong to $\weights$, where 
\[
p(a)(x) = (1+|x|)^{-a}, \qquad e(l)(x) = e^{-l|x|^{\sigma}}.  \]
Moreover, we fix
functions $\vr, \chi$ in $\mS_{\omega}$ supported in an annulus and a ball
respectively, such that for $\vr_{-1} = \chi$ and $\vr_{j}(\cdot) = \vr(2^{-j}
\cdot)$, $j \in \NN_0$, the sequence $\{ \vr_j \}_{j \ge {-}1}$ forms a dyadic
partition of the unity. We also assume that $\supp(\chi), \supp(\vr) \subset
({-}1/2, 1/2)^d$ and write $j_n \in \NN$ for the smallest index such that
$\supp(\vr_j) \not\subseteq n[{-}1/2, 1/2]^d$. For $j < j_n$ and $\varphi\colon \znd
\to \RR$ we define the \emph{Littlewood-Paley blocks}\[ \Delta_j^n \varphi = \mF_n^{-1}\big( \vr_j \mF_n(\varphi)
\big), \qquad \Delta^n_{j_n} \varphi = \mF_n^{-1}\big( (1 {-} \!\!\! \sum_{{-}1 \le j <
j_n}\vr_j) \mF_n(\varphi) \big). \] For $\alpha \in \RR, \ p,q \in [1,
\infty]$ and $z \in \weights$ we define the discrete weighted Besov spaces
$B^{\alpha}_{p,q}(\znd, z)$ via the norm: \[ \| \varphi
  \|_{B^{\alpha}_{p,q}(\znd, z)} = \big\| (2^{j\alpha} \| \Delta^n_j \varphi
\|_{L^p(\znd, z)} \big)_{j \le j_n}\|_{\ell^q(\le j_n)} \] where
$\|\varphi\|_{L^p(\znd, z)} = \big(\sum_{x \in \znd}
n^{-d}|z(x)\varphi(x)|^p\big)^{1/p}$ and $\| \cdot \|_{\ell^q(\le j_n)}$ is the
classical $\ell^q$ norm with the sum truncated at the $j_n$-th term. We write
$\mC^{\alpha}(\znd , z) := B^{\alpha}_{\infty, \infty}(\znd, z)$ and
$\mC^{\alpha}_p(\znd, z) := B^{\alpha}_{p, \infty}(\znd, z).$ The same
definitions and notations are assumed for the classical Besov spaces $B^{\alpha}_{p, q}(\RR^d, z)$, which are defined analogously (with $\Delta_j \varphi = \Delta^\infty_j \varphi = \mF_{\RR^d}^{-1}(\rho_j \mF_{\RR^d} \varphi)$ for all $j \ge -1$, and $j_\infty = \infty$). We
also consider the extension operator $\mE^n\colon B^{\alpha}_{p, q}(\znd, z) \to
B^{\alpha}_{p, q}(\RR^d,z)$ as in \cite[Lemma~2.24]{MartinPerkowski2017}. 

\begin{remark}
	In this setting we can decompose the (for $n = \infty$ a priori ill-posed) product of two distributions as \(\varphi \cdot \psi = \varphi \para \psi {+} \varphi \reso \psi
  {+} \psi \para \varphi\), with: 
  \[
    \varphi \para \psi = \sum_{ 1 \leq i \leq j_{n}} \Delta^n_{<i {-} 1} \varphi
    \Delta^{n}_{i} \psi, \ \ \varphi \reso \psi = \sum_{ \substack{|i {-} j|
    \leq 1 \\ {-} 1 \leq i, j \leq j_{n}}} \Delta^n_{i} \varphi \Delta^n_{j}
    \psi,
  \] 
  where \(\Delta^n_{<i-1} \varphi = \sum_{{-} 1 \leq j <  i {-} 1} \Delta^n_j
  \varphi\). Here we explicitly allow the case $n=\infty$. For simplicity, we do
  not include $n$ in the notation for $\para$ and $\reso$. We call $\varphi
  \para \psi$ the \emph{paraproduct}, and $\varphi \reso \psi$ the
  \emph{resonant product}.
\end{remark}

Now we consider time-dependent functions. Fix a time horizon $T>0$
and assume we are given an increasing family of normed spaces $X = (X(t))_{t \in
[0,T]}$ with decreasing norms ($X(t) \equiv X(0)$ is allowed). Usually we will
use this to deal with time-dependent weights and take $X(t) = \mC^\alpha(\znd,
e(l+t))$ for some $\alpha,l \in \RR$. We then write $CX$ for the space of
continuous functions $\varphi\colon [0,T] \to X(T)$ endowed with the supremum
norm $\| \varphi \|_{CX} = \sup_{ t \in [0, T]} \| \varphi(t) \|_{X(t)}$. For
$\alpha \in (0,1)$ we sometimes quantify the time regularity via $C^\alpha X =
\{f \in CX: \|f\|_{C^\alpha X} < \infty \}$, where
\[
    \|f\|_{C^\alpha X} = \|f\|_{CX} + \sup_{0 \le s < t \le T} \frac{\|f(t) {-}
    f(s)\|_{X(t)}}{|t{-}s|^\alpha}.
\]
To control a blowup of the norm of order $\gamma \in [0,1)$ as $t\to 0$ we also
define the spaces $\mM^{\gamma}X$ of functions $f \colon (0,T] \to X(T)$ with
norm $\| \varphi \|_{\mM^\gamma X} = \sup_{ t \in (0, T]} t^\gamma \| \varphi(t)
\|_{X(t)}$. Finally, we need the spaces \(\mL^{\gamma, \alpha}_p(\znd, e(l))\)
(see \cite[Definition~3.8]{MartinPerkowski2017}) of functions \(f \in C([0,T],
\mS_\omega') \) such that:
\[
  f \in \mM^\gamma \mC^\alpha_p(\znd, e(l+\cdot)) \ \ \text{ and } \ \  t\mapsto
  t^\gamma f(t) \in C^{\alpha/2} L^p(\znd, e(l+\cdot)).
\]
For simplicity, we will denote with $\mL^{\alpha}(\znd, e(l))$ the space $\mL^{0,\alpha}_\infty (\znd, e(l))$.
We will write \(\mfL^n = \partial_t {-} \Delta^n\), where \(\Delta^n\) is the
discrete Laplacian (for \(x,y \in \znd\) we say \(x \sim y\) if \(|x {-} y| =
n^{{-} 1}\)): \[ \Delta^n \varphi (x) = \frac{1}{n^2} \sum_{y \sim x}
(\varphi(y) {-} \varphi(x)), \] and $\Delta^\infty = \Delta$ is the usual
Laplacian. We stress that $\Delta^n$ without subscript always denotes the
discrete Laplacian, while $\Delta^n_j$ always denotes a Littlewood-Paley block.
The following estimates will be useful in the discussion ahead.

\begin{lemma}\label{lem:weighted_paraproduct_estimates} The estimates below hold
  uniformly over $n \in \NN \cup\{\infty\}$ (recall that $\ZZ^d_\infty = \RR^d$). Consider \(z, z_1,
    z_2,z_3 \in \weights\) and \(\alpha, \beta \in \mathbb{R}\). We find that:
    \begin{align*} & \| \varphi \para \psi \|_{\mC^{\alpha}_{p}(\znd; z_1 z_2)}
      \lesssim \| \varphi \|_{L^p(\znd ; z_1)} \| \psi \|_{\mC^{\alpha}(\znd ;
      z_2)}, \\ & \| \varphi \para \psi \|_{\mC^{\alpha {+} \beta}_{p}(\znd ;
      z_1 z_2)} \lesssim \| \varphi \|_{\mC^{\beta}_{p}(\znd ; z_1)} \| \psi
      \|_{\mC^{\alpha}(\znd ; z_2)}, & \text{ if } \beta < 0, \\ & \|
      \varphi\reso \psi \|_{\mC^{\alpha {+} \beta}_p (\znd; z_1 z_2)} \lesssim
    \| \varphi \|_{\mC^{\beta}_{p}(\znd ; z_1)} \| \psi \|_{\mC^{\alpha}(\znd ;
  z_2)} & \text{ if } \alpha {+} \beta >0 .  \end{align*}
  Similar bounds hold if we estimate $\psi$ in a $\mC_p$ Besov space and $\varphi$ in $\mC = \mC_\infty$. And for \(\gamma
    \in [0,1), \varepsilon \in [0, 2\gamma] \cap [0, \alpha), 0  < \alpha <2\) and \(\delta>0\) we can bound:
  \begin{equation}\label{eqn:loose_time_space_regularity}
	  \| \varphi \|_{\mL^{\gamma {-} \varepsilon/2, \alpha {-} \varepsilon }_{p}(\znd ; z) } \lesssim \| \varphi \|_{\mL^{ \gamma, \alpha} _{p} (\znd; z)}. 
	\end{equation} 
	Moreover, for the operator $C_1(\varphi,\psi,\zeta) = (\varphi\para \psi) \reso \zeta - \varphi ( \psi\reso \zeta)$ we have:
	\[
		\|
		C_1(\varphi,\psi,\zeta)\|_{\mC^{\beta+\gamma}_p(\znd;z_1z_2z_3)}
		\lesssim \|\varphi\|_{\mC^\alpha_p(\znd; z_1)}
		\|\psi\|_{\mC^\beta(\znd;z_2)} \|\zeta\|_{\mC^\gamma(\znd;z_3)},
	\]
	if $\beta{+}\gamma < 0$, $\alpha{+}\beta{+}\gamma>0$.
\end{lemma}

\begin{proof} 
  The first three estimates are shown in
  \cite[Lemma~4.2]{MartinPerkowski2017} and the fourth estimate comes from
  \cite[Lemma~3.11]{MartinPerkowski2017}. In that lemma the case $\varepsilon =
  2\gamma < \alpha$ is not included, but it follows by the same arguments
  (since \cite[Lemma~A.1]{Gubinelli2017KPZ} still applies in that case). The
  last estimate is provided by \cite[Lemma 4.4]{MartinPerkowski2017}.
\end{proof}

For two functions $\psi, \varphi\colon \RR^d \to \RR$ we define $\langle \psi,
\varphi \rangle = \smallint \ud x \ \psi(x) \varphi(x)$ and \(\psi* \varphi(x)=
\langle \psi(x {-} \cdot), \varphi( \cdot) \rangle\) for \(x \in \RR^{d}\), whereas if $\psi,
\varphi\colon \znd \to \RR$ we write \( \langle \psi, \varphi \rangle_n =
\frac{1}{n^d} \sum_{x \in \znd}\psi(x)\varphi(x)\) and \(\psi*_n \varphi(x) = \langle \psi(x {-} \cdot),
\varphi( \cdot) \rangle_{n}\) for \(x \in \znd\). 

Finally, for a metric space $E$ we denote with $\DD([0, T]; E)$ and $\DD([0,+\infty); E)$ the Skorohod
space equipped with the Skorohod topology (cf. \cite[Section 3.5]{EthierKurtz1986}).
We will also write $\mM(\RR^d)$ for
the space of positive finite measures on $\RR^d$ with the weak topology, which is
a Polish space (cf. \cite[Section 3]{DawsonMaisonneuve1993SaintFlour}).

\section{The Model} \label{sectn:model}

We consider a branching random walk in a random environment
(BRWRE). This is a process on the lattice $\znd$, for $n \in \NN$ and $d = 1,2$, and we are interested in the limit $n \to \infty$. The
evolution of this process depends on the environment it lives in. Therefore, we
first discuss the environment before introducing the Markov process.

A \emph{deterministic environment} is a sequence $\{\xi^n\}_{n \in \NN}$ of potentials
on the lattice, i.e. functions $\xi^n\colon \znd \to \RR.$ 
A \emph{random environment} is a sequence of
probability spaces \( (\Omega^{p, n}, \mF^{p, n}, \PP^{p, n})\) together with a
sequence $\{\xi^n_p\}_{n \in \NN}$ of measurable maps $\xi^n_p \colon
\Omega^{p,n} \times \znd \to \RR$.

\begin{assumption}[Random Environment]\label{assu:noise} We assume that for
  every \(n \in \NN\),  $\{\xi^n_p(x)\}_{x \in \ZZ^d_n}$ is a set of i.i.d. random
  variables on a probability space \((\Omega^{p, n}, \mF^{p, n}, \PP^{p, n})\) which satisfy: 
  \begin{equation}\label{eqn:distr_xi} n^{-d/2}
    \xi^n_p(x) = \Phi \ \ \text{in distribution},
  \end{equation} for a random variable $\Phi$
 with finite moments of every order such that \[\EE[\Phi] = 0, \ \
\EE[\Phi^2] = 1.\]
\end{assumption}

\begin{remark} 
  It follows that $\xi_p^n$ converges
  in distribution to a white noise $\xi_p$ on $\RR^d$, in the sense that \(
  \langle \xi^n_p, f\rangle_n \to \xi_p(f) \) for all $f \in C_c(\RR^d)$.
\end{remark}

To separate the randomness coming from the potential from that of the branching
random walks it will be convenient to freeze the realization of $\xi^n_p$ and to
consider it as a deterministic environment. But we cannot expect to
obtain reasonable scaling limits
for all deterministic environments. Therefore, we need to identify properties that hold for typical realizations of random potentials
satisfying Assumption~\ref{assu:noise}. The reader only interested in random
environments may skip the following assumption and use it as a black box, since
by Lemma~\ref{lem:renormalisation} below it is satisfied under Assumption~\ref{assu:noise}.

\begin{assumption}[Deterministic environment]\label{assu:renormalisation}
  Let $\xi^n$ be a deterministic environment and let $X^n$ be the solution to the equation ${-}\Delta^n X^n =
\chi(D)\xi^n = \mF_n^{-1}(\chi \mF_n \xi^n)$ in the sense explained in
\cite[Section~5.1]{MartinPerkowski2017}, where $\chi$ is a smooth function equal
to $1$ outside of $(-1/4,1/4)^d$ and equal to zero on $(-1/8,1/8)^d$. Consider
a regularity parameter
	\[
		\alpha \in (1, \tfrac32) \text{ in } d = 1, \qquad \alpha \in (\tfrac23, 1) \text{ in } d=2.
	\]
 We assume that the following holds:
  \begin{itemize}
	\item[(i)] There exists $\xi \in \bigcap_{a>0} \mC^{\alpha{-}2}(\RR^d, p(a))$ such that for all
	  $a>0$: \[ \sup_n
	  \|\xi^n\|_{\mC^{\alpha{-}2}(\znd, p(a))}< {+}\infty \ \text{ and } \
	\mE^n \xi^n \to \xi \text{ in } \mC^{\alpha{-}2}(\RR^d, p(a)).\] 

	\item[(ii)]  For any $a, \ve >0$ we can bound: \[ \sup_n \|n^{-d/2} \xi^n_{+}
	  \|_{\mC^{-\ve}(\znd, p(a))} + \sup_n \|n^{-d/2} |\xi^n|
	\|_{\mC^{-\ve}(\znd, p(a))} < {+}\infty \] as well as for any $b> d/2$:
	\[ \sup_n \|n^{-d/2} \xi^n_{+} \|_{L^2(\znd, p(b))}  < {+}\infty. \]
	Moreover, there exists $\nu \ge 0$ such that the following convergences
	hold: \begin{align*} 	\mathcal{E}^n n^{-d/2}\xi^n_+  \to \nu,
	\qquad \mathcal{E}^n n^{-d/2}|\xi^n|  \to 2\nu \end{align*} in
	$\mC^{-\ve}(\RR^d, p(a))$.

	\item[(iii)] If $d = 2$ there exists a sequence $\{c_n\} \subset \RR$ with
	  $n^{-d/2}c_n \to 0$ and there exist $X \in \bigcap_{a > 0} \mC^\alpha(\RR^d, p(a))$ and $X\diamond\xi \in \bigcap_{a>0}
	  \mC^{2\alpha{-}2}(\RR^d, p(a))$ which satisfy for all $a>0$: \[\sup_n \|
	  X^n\|_{\mC^{\alpha}(\znd, p(a))} + \sup_n \| (X^n \reso \xi^n) {-}c_n
	\|_{\mC^{2\alpha{-}2}(\znd,p(a))} <{+} \infty \] and $\mE^n X^n \to X$
	in $\mC^{\alpha}(\RR^d, p(a))$ and $\mE^n\big( (X^n \reso\xi^n) {-}c_n
	\big) \to X\diamond \xi$ in $\mC^{2\alpha{-}2}(\RR^d, p(a))$.
\end{itemize} \end{assumption}

We say that $\xi \in \mS_\omega'(\RR^d)$ is a \emph{deterministic
environment satisfying Assumption~\ref{assu:renormalisation}} if there exists a
sequence $\{\xi^n\}_{n \in \NN}$ such that the conditions of
Assumption~\ref{assu:renormalisation} hold.

The next result establishes the connection between the probabilistic and the
analytical conditions. To formulate it we need the following
sequence of diverging \emph{renormalization} constants:
\begin{equation}\label{eqn:real_line_renormalisation_constant} 
  \kappa_n = \int_{\TT^2_n} \ud k \ \frac{\chi(k)}{l^n(k)} \sim \log(n), 
\end{equation} 
with $l^n$ being the Fourier multiplier associated to the discrete Laplacian $\Delta^n$ and $\chi$ as in Assumption~\ref{assu:renormalisation}.

\begin{lemma}\label{lem:renormalisation} 
  Given a random environment $\{ \bar{\xi}^n_p \}_{n \in \NN}$ satisfying
  Assumption \ref{assu:noise}, there exists a probability space $(\Omega^p,
  \mathcal{F}^p, \PP^p)$ supporting random variables $\{\xi_p^n\}_{n \in \NN}$
  such that $\bar{\xi}_p^n = \xi^n_p$ in distribution and such that $\{
  \xi^n_p(\omega^p, \cdot)\}_{n \in \NN}$ is a deterministic environment
  satisfying Assumption \ref{assu:renormalisation} for all $\omega^p \in
  \Omega^p.$ Moreover the sequence $c_n$ in Assumption
  \ref{assu:renormalisation} can be chosen equal to $\kappa_n$ (see Equation
  \eqref{eqn:real_line_renormalisation_constant}) outside of a null set.
  Similarly, $\nu$ is strictly positive and deterministic outside of a null set
and equals the expectation $\EE[ \Phi_+]$.  
\end{lemma}

\begin{proof} 
  The existence of such a probability space is provided by the Skorohod
  representation theorem. Indeed it is a consequence of Assumption
  \ref{assu:noise} that all the convergences hold in the sense of distributions:
  The convergences in (i) and (iii) follow from
  Lemma~\ref{lem:convergence_1d_white_noise} if $d=1$ and from \cite[Lemmata~5.3
  and~5.5]{MartinPerkowski2017} if $d=2$ (where it is also shown that we can
  choose $c_n = \kappa_n$). The convergence in (ii) for $\nu = \EE[\Phi_+]$ is
  shown in Lemma~\ref{lem:mean_noise}. After changing the probability space the
  Skorohod representation theorem guarantees almost sure convergence, so setting
  $\xi^n, \xi,  c^n, \nu = 0$ on a null set we find the result for every
  $\omega^p$. (There is a small subtlety in the application of the Skorohod
  representation theorem because $\mC^{\gamma}(\RR^d,p(a))$ is not separable, but we can restrict our attention to the closure
of smooth compactly supported functions in $\mC^{\gamma}(\RR^d,p(a))$, which is
a closed separable subspace).
\end{proof}

\begin{notation}
  A sequence of random variables $\{\xi^n_p\}_{n \in \NN}$ defined on a common
  probability space \( (\Omega^{p}, \mF^{ p}, \PP^{p})\) which almost surely
  satisfies Assumption~\ref{assu:renormalisation} is called a \emph{controlled
  random environment}. By Lemma~\ref{lem:renormalisation}, for any random
  environment satisfying Assumption~\ref{assu:noise} we can find a controlled
  random environment with the same distribution. For a given controlled random
  environment we introduce the effective potential: 
  \[ 
    \xi^{n}_{p,e}(\omega^p, x) = \xi^n_p(\omega^p, x){-}c_n(\omega^p) 1_{\{d =
    2\}}.  
  \] 
  Given a controlled random environment we define
  $\mH^{\omega^p}$ as the random Anderson Hamiltonian and its
  domain $\mD_{\mH^{\omega^p}}$ (see Lemma~\ref{lem:Anderson-hamiltonian}). If
  the environment is deterministic we drop all indices \(p\).
\end{notation}

We pass to the description of the particle system. This will be a (random)
Markov process on the space $E = \big( \NN^{\znd}_0 \big)_0$ of compactly supported functions $\eta\colon
\znd \to \NN_0$, whose construction is discussed in
Appendix \ref{sectn:construction_markov_process}. We define $\eta^{x \mapsto
y}(z) = \eta(z) {+} (1_{\{y\}}(z) {-} 1_{\{x\}}(z)) 1_{ \{\eta(x) \geq 1\}}$ and
$\eta^{x \pm }(z) = (\eta(z) \pm 1_{\{ x \}}(z))_+$. Moreover, $C_b(E)$ is the
Banach space of continuous and bounded functions on $E$ endowed with the
discrete topology. For \(F \in C_{b}(E), x \in \znd\) we write:
\[ \Delta^{n}_xF (\eta) = n^{2} \sum_{y \sim x} (F( \eta^{x \mapsto y}) {-}
F( \eta)), \qquad d^{ \pm}_{x} F(\eta) = F(\eta^{x \pm}) {-} F( \eta).\] 

\begin{definition}\label{def:of_u_n_and_mu_n} 
  Fix an ``averaging parameter'' $\vr \ge 0$ and a controlled random environment
  $\xi^n_p$. Let $\PP^n$ be the measure on $\Omega^p \times \DD([0, {+}\infty);
  E)$ defined as the ``semidirect product measure'' (cf. \eqref{eqn:definition_semidirect_product_measure}) $\PP^{p} \ltimes
  \PP^{\omega^p, n}$, where for $\omega^p \in \Omega^p$ the measure
  $\PP^{\omega^p, n}$ on $\DD([0, {+}\infty); E)$ is the law under which the
  canonical process $u^n_p(\omega^p, \cdot)$ started in $u^n_p(\omega^p, 0) =
  \lfloor n^\vr  \rfloor 1_{\{ 0\}}(x)$ is the Markov process with
  generator $$\mL^{n, \omega^p} \colon \mD(\mL^{n, \omega^p}) \to C_b(E),$$
  where \(\mL^{n, \omega^{p}}(F)(\eta)\) is defined by: 
  \begin{equation}\label{eqn:defn_generator}
       \sum_{x \in \ZZ_n^d} \eta_x  \cdot \Big[ \Delta^{n}_{x}F (\eta) +
	 (\xi^n_{p, e})_{+}(\omega^p, x) d^{+}_x F( \eta)+
       (\xi^n_{p,e})_{-}(\omega^p, x) d^{-}_x F( \eta) \Big]
  \end{equation} 
  and the domain $\mD (\mL^{n, \omega^p})$ consists of all $F\in
  C_b(E)$ such that the right-hand side of \eqref{eqn:defn_generator}
  lies in $C_b(E)$. To $u^n_p$ we associate the process \(\mu^{n}_{p}\) with the
  pairing \[\mu^n_p(\omega^p, t)(\varphi):= \sum_{x \in  \mathbb{Z}_n^d}
  \lfloor n^{\vr} \rfloor^{-1} u^n_p(\omega^p, t,x) \varphi(x)\] for any
  function $\varphi\colon \znd \to \RR$. Hence $\mu^n_p$ is a stochastic
  process with values in $\DD([0, +\infty); \mM(\RR^d))$, with the law induced
  by $\PP^{n}$.
\end{definition}

\begin{remark}
	Although not explicitly stated, it is part of the definition that
	$\omega^p \mapsto \PP^{\omega^p,n}(A)$ is measurable for Borel sets $A
	\in \mB(\DD([0,+\infty);E))$.	
\end{remark}

Since all particles evolve independently, we expect that for
	$\vr \to \infty$ the law of large numbers applies. This is why we refer to $\vr$ as
	an averaging parameter.

\begin{notation} 
  In the terminology of stochastic processes in random media, we
  refer to $\PP^{\omega^p, n}$ as the \emph{quenched law} of the process
  $u^n_p$ (or $\mu^n_p$) given the noise $\xi^n_p$. We also call $\PP^{n}$ the
  \emph{total law}. As before, if the process is deterministic we drop the index
  \(p\) everywhere.
\end{notation}

We can now state the main convergence results of this work. We will first prove quenched
versions and the total versions are then easy
corollaries. We start with a law of large numbers.

\begin{theorem}\label{thm:LLN} 
  Let \(\xi\) be a deterministic environment satisfying
  Assumption \ref{assu:renormalisation} and let $\vr > d/2$.
  Let $w$ be the solution of PAM \eqref{eqn:pam} with initial condition $w(0,x)
  = \delta_0(x)$, as constructed in Proposition~\ref{prop:convergence_PAM} (cf.
  Remark \ref{rem:PAM_Green_function}). The measure-valued process $\mu^n$
  from Definition \ref{def:of_u_n_and_mu_n} converges to $w$ in probability in
  the space $\DD([0, +\infty); \mM(\RR^d))$ as $n \to +\infty$.
\end{theorem}

\begin{proof} 
  The proof can be found in Section \ref{sectn:convergence}.
\end{proof}

If the averaging parameter takes the critical value $\vr = d/2$, we see random
fluctuations in the limit and we end up with the \emph{rough super-Brownian
motion} (rSBM). As in the case of the
classical SBM, the limiting
process can be characterized via duality with the following equation:
\begin{equation}\label{eqn:PDE_laplace_duality_continuous} \partial_t\varphi =
\mH \varphi {-} \frac{\kappa}{2} \varphi^2, \qquad \varphi(0) = \varphi_0,
\end{equation} for $\varphi_0 \in C_c^{\infty}(\RR^d)$, $\varphi_0 \ge 0$, where we recall that $\mH$ is the Anderson Hamiltonian. With
some abuse of notation (since the equation is not linear) we write $U_t
\varphi_0 = \varphi(t)$. 

\begin{definition}\label{def:rSBM}
   Let $\xi$ be a deterministic environment satisfying
   Assumption~\ref{assu:renormalisation}, let $\kappa>0$ and let $\mu$ be a
   process with values in the space $C([0, +\infty);\mM(\RR^d))$, such that
   $\mu(0) = \delta_0$.  Write $\mF = \{ \mF_t\}_{t \in [0, +\infty)}$ for the
   completed and right-continuous filtration generated by $\mu$. We call
   $\mu$ a \emph{rough super-Brownian motion (rSBM) with parameter $\kappa$} if it
   satisfies one of the three properties below:
  \begin{enumerate}[label=(\roman*)] 
    
    \item For any $t \ge 0$ and $\varphi_0 \in
      C_c^{\infty}(\RR^d), \varphi_0 \ge 0$ and for $U_{\cdot} \varphi_0$ the
      solution to Equation \eqref{eqn:PDE_laplace_duality_continuous} with
      initial condition $\varphi_0$, the process 
      \[ N^{\varphi_0}_t(s) = e^{-\langle \mu(s), U_{t {-} s} \varphi_0 \rangle
      }, \qquad s \in [0, t], \] 
	is a bounded continuous $\mF-$martingale.
		
	\item For any $t \ge 0$ and $\varphi_0 \in C^{\infty}_c(\RR^d)$ and $f
	  \in C([0, t];\mC^{\zeta}(\RR^d,e(l)))$ for some $\zeta >0$ and $l<
	  -t$, and for $\varphi_t$ solving \[ \partial_s \varphi_t + \mH
	  \varphi_t = f, \qquad s \in [0, t], \qquad \varphi_t(t) = \varphi_0,
	\] it holds that \begin{align*} s\mapsto M_t^{\varphi_0, f}(s) : = \langle \mu(s)
	, \varphi_t(s) \rangle{-} \langle \mu(0), \varphi_t(0) \rangle {-}
      \int_0^s \ud r \ \langle \mu(r), f(r) \rangle,
    \end{align*} defined for $s \in [0, t]$, is a continuous square-integrable $\mF-$martingale with
    quadratic variation \[ \langle M^{\varphi_0,f}_t\rangle_s = \kappa \int_0^s
    \ud r \ \langle \mu(r), (\varphi_t)^2(r) \rangle .  \] 
		 
	\item For any $\varphi \in \mD_{\mH}$ the process: \[ L^{\varphi}(t) =
	  \langle \mu(t), \varphi \rangle{-} \langle \mu(0), \varphi \rangle
	{-}\int_0^t \ud r \ \langle \mu(r), \mH \varphi \rangle, \qquad t \in
      [0,+\infty), \] is a continuous 	$\mF-$martingale, square-integrable on
      $[0,T]$ for all $T > 0$, with quadratic variation \[ \langle L^{\varphi}
      \rangle_t = \kappa  \int_0^t  \ud r \ \langle \mu(r), \varphi^2 \rangle.
    \]
	\end{enumerate} 
\end{definition}

Each of the three properties above characterizes
the process uniquely:

\begin{lemma}\label{lem:rSBM-uniqueness-equivalence}

  The three conditions of Definition \ref{def:rSBM} are equivalent. Moreover,
  if $\mu$ is a rSBM with parameter $\kappa$, then its
  law is unique.

\end{lemma}

\begin{proof}
	The proof can be found at the end of Section
	\ref{sectn:convergence}.	
\end{proof}

\begin{theorem}\label{thm:CLT} 

  Let $\{\xi^n\}_{n \in \NN}$ be a deterministic environment satisfying
  Assumption \ref{assu:renormalisation} and let $\vr =
  d/2$. Then the sequence $\{\mu^n\}_{n \in \NN}$ converges to the 
  rSBM $\mu$ with parameter $\kappa = 2\nu$  in distribution in $\DD([0, +\infty);
  \mM(\RR^d))$. 

\end{theorem}

\begin{proof} 
  The proof can be found at the end of Section \ref{sectn:convergence}.
\end{proof}

\begin{remark}
	Lemma~\ref{lem:rSBM-uniqueness-equivalence} gives the uniqueness of the
	rSBM for all parameters $\kappa > 0$, but Theorem \ref{thm:CLT} only shows the
	existence conditionally on the existence of an environment which satisfies
	Assumption~\ref{assu:renormalisation}, which leads to the constraint $\nu
	\in (0,\frac12]$ because we should think of $\nu = \EE[\Phi_+]$
	for a centered random variable $\Phi$ with $\EE[\Phi^2]=1$.  But we 
	establish the existence of the rSBM for general $\kappa > 0$
	in Section~\ref{sectn:mixing_dawson_watanabe}.
\end{remark}

\begin{remark}
  We restrict our attention to the Dirac delta initial condition for
	simplicity, but most of our arguments extend to initial
	conditions $\mu \in
	\mM(\RR^d)$ that satisfy $\langle \mu, e(l)\rangle < \infty$ for all $l
	< 0$. In this case only the construction of the initial value sequence
	\(\{\mu^{n}(0)\}_{n \in \NN}\) is more technical, because we
	need to come up with an approximation in terms of integer valued point
      measures (which we need as initial condition for the particle system). This can be achieved by discretizing the initial measure on a coarser grid.
\end{remark}

The previous results describe the scaling behavior of the BRWRE
conditionally on the environment, and we now pass to the unconditional statements. To
a given random environment $\xi^n_p$ satisfying Assumption \ref{assu:noise} (not
necessarily a \textit{controlled} random environment) we associate a sequence
of random variables in $\mS_{\omega}^\prime(\RR^d)$ by defining $\xi^n_p(f) = n^{-d}
\sum_x \xi^n_p(x) f(x)$. The sequence of measures $\overline{\PP}^n =
\PP^{p, n} \ltimes \PP^{\omega^p, n}$ on $\mS_{\omega}^\prime(\RR^d) \times
\DD([0, +\infty); \mM(\RR^d))$ is then such that $\PP^{p, n}$ is the law of $\xi^n_p$
and $\PP^{\omega^p, n}$ is the quenched law of the branching process $\mu^n_p$
given $\xi^n_p$ (cf. Appendix \ref{sectn:construction_markov_process}).

\begin{corollary}\label{cor:convergence_total_law} 
  The sequence of measures $\overline{\PP}^n$ converges weakly to 
  $\overline{\PP} = \PP^p \ltimes \PP^{\omega^p}$ on
  $\mS_{\omega}^\prime(\RR^d)\times \DD([0, +\infty); \mM(\RR^d))$, where $
  \PP^p$ is the law of the space white noise on $\mS_{\omega}^\prime(\RR^d)$,
  and $\PP^{\omega^p}$ is the quenched law of $\mu_p$ given $\xi_p$ which is
  described by Theorem \ref{thm:LLN} if $\vr>d/2$ or by Theorem \ref{thm:CLT} if
  $\vr = d/2$.  
\end{corollary}

\begin{proof} 
  Consider a function $F$ on $\mS_{\omega}^\prime(\RR^d) \times \DD([0,
  +\infty); \mM(\RR^d))$ which is continuous and bounded. We
  need the convergence \(\lim_n \EE \big[ F(\xi^n_p, \mu^n) \big] \to
  \EE \big[ F(\xi_p, \mu)\big]\). Up to changing the probability space (which
    does not affect the law) we may assume that $\xi^n_p$ is a controlled random
    environment. We condition on the noise, rewriting the left-hand side as
  \[
	  \EE \big[ F(\xi^n_p, \mu^n) \big] = \int  \EE^{\omega^p,n}\big[F(\xi^n_p(\omega^p),\mu^n)\big] \PP^p(\operatorname{d} \omega^p) .
  \]
  Under the additional property of being a controlled random environment and for
  fixed \(\omega^p \in \Omega^p\), the conditional law $\PP^{\omega^p, n}$ on
  the space $\DD([0, +\infty); \mM(\RR^d))$ converges weakly to the measure
  $\PP^{\omega^p}$ given by Theorem~\ref{thm:LLN} respectively
  Theorem~\ref{thm:CLT}, according to the value of $\vr$.  We can thus deduce
  the result by dominated convergence.  
\end{proof}

For $\vr > d/2$ the process of Corollary~\ref{cor:convergence_total_law} is
simply the continuous parabolic Anderson model. For $\vr = d/2$ it is a new
process.

\begin{definition}
	For $\vr = d/2$ we call the process $\mu$ of
	Corollary~\ref{cor:convergence_total_law} an \emph{SBM in static random
	environment} (of parameter $\kappa > 0$).
\end{definition}

In dimension $d = 1$ we characterize the process $\mu$ as the solution to the
SPDE \eqref{eqn:SPDE_rsbm}. First, we rigorously define solutions to such an
equation.

\begin{definition}\label{def:solutions_to_SPDE_1D}
  
  Let $d = 1$, $\kappa>0$, and $\pi \in \mathcal M
  (\RR)$. A weak solution to
  \begin{align}\label{eqn:spde_rsbm_formal_with_kappa}
    \partial_t \mu_p(t,x) = \mH^{\omega^p} \mu_p(t,x) {+} \sqrt{\kappa \mu_p(t,x)} \tilde{\xi}(t,x),\qquad \mu_p(0) = \pi,
  \end{align}
  is a couple formed by a probability space $(\Omega, \mF, \PP)$ and a random
  process $$\mu_p\colon \Omega \to C([0, +\infty); \mM(\RR))$$ such that $\Omega
  = \Omega^p \times \bar{\Omega}$ and $\PP$ is of the form $\PP^p \ltimes
  \PP^{\omega^p}$ with $(\Omega^p, \PP^p)$ supporting a space white noise
  $\xi_p$ and $(\Omega, \PP)$ supporting an independent space-time white noise $\tilde{\xi}$, such that the following properties are
  fulfilled for almost all $\omega^p \in \Omega^p$: 
  \begin{itemize} 
    \item There exists a filtration $\{\mF^{\omega^p}_t\}_{t \in
      [0, T]}$ on the space $(\bar{\Omega}, \PP^{\omega^p})$ which satisfies the
      usual conditions and such that $\mu_p(\omega^p, \cdot)$ is adapted and almost surely lies in $L^{p}([0, T]; L^2(\RR,e(l)))$ for
      all $p < 2$ and $l \in \RR$.  Moreover, under $\PP^{\omega^p}$ the process
      $\tilde{\xi}(\omega^p, \cdot)$ is a space-time white noise adapted to the
      same filtration.  
    \item The random process $\mu_p$ satisfies for all $\varphi \in
      \mD_{\mH^{\omega^p}}$ and for all \(t \geq 0\): 
      \begin{align*} \int_\RR  \ud x \ & \mu_p(\omega^p, t, x) \varphi(x)  = 
	\int_0^t \int_\RR \ud s \ud x \
	\mu_p(\omega^p, s, x) (\mH^{\omega^p} \varphi)(x) \\ & {+}  \int_0^t
	\int_\RR \tilde{\xi}(\omega^p, \ud s, \ud x)\sqrt{\kappa \mu_p(\omega^p,
	s, x)}\varphi(x) + \int_\RR \varphi(x) \pi(dx),
      \end{align*} 
      with the last integral understood in the sense of Walsh
\cite{Walsh1986}.  \end{itemize}
  
  \end{definition}


\begin{theorem}
  
  \label{thm:rsbm_SPDE} For $\pi = \delta_0$ and any $\kappa >0$ there exists a weak
  solution $\mu_p$ to the SPDE \eqref{eqn:spde_rsbm_formal_with_kappa} in the
  sense of Definition~\ref{def:solutions_to_SPDE_1D}. The law of $\mu_p$ as a random process on
  $C([0, +\infty); \mM(\RR))$ is unique and it corresponds to an SBM in static
random environment of parameter \(\kappa\).  

\end{theorem} 

\begin{proof} 
  The proof can be found at the end of Section \ref{secnt:SPDE_derivation}.  
\end{proof}

As a last result, we show that the rSBM is persistent in dimension $d=1,2$.

\begin{definition}\label{def:persistency} 
  
  We say that a random process $\mu \in C([0, {+}\infty); \mM(\RR^d))$ is
  \emph{super-exponentially persistent} if for any nonzero positive function
  $\varphi \in C_c^{\infty}(\RR^d)$ and for all $\lambda > 0$ it holds that: \[
    \PP \big( \lim_{t \to \infty} e^{-t \lambda} \langle \mu(t) , \varphi
  \rangle = \infty\big) > 0.\] 

\end{definition}

\begin{theorem}\label{thm:persistence}  
  Let $\mu_p$ be an SBM in static random environment. Then for almost all
  $\omega^p \in \Omega^p$ the process $\mu_p(\omega^p, \cdot)$ is
  super-exponentially persistent.
\end{theorem}

The result follows from Corollary \ref{cor:persistence_rsbm} and the preceding
discussion.

\section{Discrete and Continuous PAM \& Anderson Hamiltonian}\label{sectn:pam}

 Here we review the solution theory for the PAM \eqref{eqn:pam} in the discrete
 and continuous setting and the interplay between the two. 

 Recall that the regularity parameter $\alpha$ from
 Assumption~\ref{assu:renormalisation} satisfies:
\begin{equation}\label{eqn:alpha_bounds}
	\alpha \in (1, \tfrac32) \text{ in } d = 1, \qquad \alpha \in (\tfrac23, 1) \text{ in } d=2. 
\end{equation}

We recall some results from \cite{MartinPerkowski2017} regarding the
solution of the PAM on the whole space (see also~\cite{Hairer2015Simple}), and
regarding the convergence of lattice models to the PAM. We take an initial
condition $w_0 \in \mC^\zeta_p(\RR^d,e(l))$ and a forcing $f \in
\mM^{\gamma_0}\mC^{\alpha_0}_p(\RR^d, e(l))$, and we consider the 
equation
\begin{equation}\label{eqn:pam-rigorous-continuous}
	\partial_t w = \Delta w + \xi w + f,\qquad w(0) = w_0
\end{equation}
and its discrete counterpart
\begin{equation}\label{eqn:pam-rigorous-discrete}
	\partial_t w^n = (\Delta^n + \xi^n_e) w^n + f^n, \qquad w^n(0) = w^n_0.
\end{equation}

To motivate the constraints on the parameters appearing in the proposition
below, let us first formally discuss the solution theory in $d=1$. Under
Assumption~\ref{assu:renormalisation} it follows from the Schauder estimates in
\cite[Lemma~3.10]{MartinPerkowski2017} that the best regularity we can expect at
a fixed time is $w(t) \in \mC^{\alpha \wedge (\zeta {+} 2) \wedge
(\alpha_0{+}2)}_p(\RR,e(k))$ for some $k \in \RR$. In fact we lose a bit of
regularity, so let $\vartheta < \alpha$ be ``large enough'' (we will see soon
what we need from $\vartheta$) and assume that $\zeta + 2 \ge \vartheta$ and
$\alpha_0 + 2 \ge \vartheta$. Then we expect $w(t) \in \mC^\vartheta_p(\RR,
e(k))$, and the Schauder estimates suggest the blow-up $\gamma =
\max\{(\vartheta+\varepsilon - \zeta)_+/2, \gamma_0\}$ for some $\varepsilon >
0$, which has to be in $[0,1)$ to be locally integrable, so in particular
$\gamma_0 \in [0,1)$. If $\vartheta + \alpha - 2 > 0$ (which is possible because
in $d=1$ we have $2\alpha - 2 > 0$), then the product $w(t)\xi$ is well defined
and in $\mC^{\alpha-2}_p(\RR, e(k)p(a))$, so we can set up a Picard iteration. The
loss of control in the weight (going from $e(k)$ to $e(k)p(a)$) is handled by
introducing time-dependent weights so that $w(t) \in \mC^\vartheta_p(\RR^d,
e(l+t))$. In the setting of singular SPDEs this idea was introduced by
Hairer-Labb\'e~\cite{Hairer2015Simple}, and it induces a small loss of
regularity which explains why we only obtain regularity $\vartheta < \alpha$ for
the solution and the additional $+\varepsilon/2$ in the blow-up $\gamma$.

In two dimensions the white noise is less regular, we no longer have $2\alpha
- 2 >0$, and we need paracontrolled analysis to solve the
equation. The solution lives in a space of \emph{paracontrolled distributions},
and now we take $\vartheta > 0$ such that $\vartheta + 2\alpha - 2 >0$. We now need additional regularity
requirements for the initial condition $w_0$ and for the forcing $f$. More
precisely, we need to be able to multiply $(P_t w_0)\xi$ and $\big(\int_0^t
P_{t{-}s} f(s) \ud s\big) \xi$, and therefore we require now also $\zeta {+} 2
{+}(\alpha {-} 2) > 0$ and $\alpha_0 {+} 2 {+} (\alpha {-} 2) > 0$, i.e. $\zeta,
\alpha_0 > {-} \alpha$. 

We do not provide the details of the construction and refer
to~\cite{MartinPerkowski2017} instead, where the two-dimensional case is worked
out (the one-dimensional case follows from similar, but much easier arguments).

\begin{proposition}\label{prop:convergence_PAM} Consider \(\alpha\) as in
  \eqref{eqn:alpha_bounds}, any \(T >0\), $p \in [1, {+}
  \infty]$, $l \in \RR, \gamma_{0} \in [0, 1)$ and $\vartheta, \zeta, \alpha_0$ satisfying:
  \begin{equation}\label{eqn:condition_parameters_PAM}
	  \vartheta \in \begin{cases} (2{-}\alpha, \alpha), & d=1, \\ (2{-}2\alpha, \alpha), & d = 2, \end{cases}
    \ \ \zeta >(\vartheta{-}2)\vee ({-}\alpha), \ \ \alpha_0 >(\vartheta{-}2)\vee ({-}\alpha),
  \end{equation}
  and let $w^n_0 \in \mC^{\zeta}_p(\znd, e(l))$ and $f^n \in
  \mM^{\gamma_0}\mC^{\alpha_0}_p(\znd, e(l))$ be such that \[	\mE^n w^n_0 \to
    w_0, \text{ in } \mC^{\zeta}_p(\RR^d, e(l)), \qquad \mE^n f^n \to f \text{
  in } \mM^{\gamma_0}\mC^{\alpha_0}_p(\RR^d, e(l)).  \] Then under Assumption
  \ref{assu:renormalisation} there exist unique (paracontrolled) solutions $w^n,
  w$ to Equation \eqref{eqn:pam-rigorous-discrete} and
  \eqref{eqn:pam-rigorous-continuous}.  Moreover, for all $\gamma > (\vartheta{-}\zeta)_+ / 2 \vee
  \gamma_0$ and for all \(\hat{l} \geq l {+} T\), the sequence $w^n$ is uniformly
  bounded in $\mL_p^{\gamma, \vartheta}(\znd, e(\hat{l}))$:
  \begin{equation}\label{eqn:prop_convergence_pam_a_priori_bound}
    \sup_{n} \| w^{n} \|_{\mL^{\gamma, \vt}_{p} (\znd, e( \hat{l}))} \lesssim
    \sup_{n} \| w^{n}_{0} \|_{ \mC^{\zeta}_{p}( \znd, e(l))} + \sup_{n}\|
    f^{n} \|_{ \mM^{\gamma_{0}} \mC^{\alpha_{0}}_{p}(\znd, e(l)) },
  \end{equation}
  where the proportionality constant depends on the time horizon \(T\) and the
  norms of the objects in Assumption \ref{assu:renormalisation}. Moreover 
  \[
  \mE^n w^n \to w \text{ in } \mL_p^{\gamma, \vartheta}(\RR^d, e(\hat{l})).
\]
\end{proposition}
 
\begin{remark}\label{rem:PAM_Green_function} 
  We consider the case \(p <
  \infty\) to start the equation in the Dirac measure
  \(\delta_0\). Indeed, \(\delta_{0}\) lies in \(\mC^{{-} d}
  (\RR^d, e(l))\) for any $l\in \RR$. This means that $\zeta=-d$, and in $d=1$
  we can choose $\vartheta$ small enough such
  that~\eqref{eqn:condition_parameters_PAM} holds. But in \(d = 2\) this is not sufficient, so we use
  instead that \(\delta_0 \in \mC^{d(1 {-} p)/p }_{p}(\RR^d, e(l))\) for \(p \in [1,
  \infty]\) and any \(l \in \RR\), so that for \(p \in [1, 2)\) the conditions
  in~\eqref{eqn:condition_parameters_PAM} are satisfied.
\end{remark}

\begin{notation}
  We write $$ t \mapsto T_t^n w_0^n + \int_0^t \ud s \ T^n_{t{-}s} f^n_s, \qquad
  t \mapsto T_t w_0 + \int_0^t \ud s \ T_{t{-}s} f_s $$ for the solution to
  Equation \eqref{eqn:pam-rigorous-discrete} and
  \eqref{eqn:pam-rigorous-continuous}, respectively.
\end{notation}

Proposition~\ref{prop:convergence_PAM} provides us with the tools to make sense of Property~(ii) in the definition of the rSBM, Definition~\ref{def:rSBM}. To make sense of the last Property~(iii), we need to construct the Anderson Hamiltonian. In finite volume this was done in~\cite{Fukushima1976, Allez2015, Gubinelli2018Semilinear, Labbe2018}, respectively, but the construction in infinite volume is more complicated, for example because the spectrum of $\mH$ is unbounded from above and thus resolvent methods fail. Hairer-Labb\'e~\cite{Hairer2018Multiplicative} suggest a construction based on
spectral calculus, setting $\mH = t^{-1} \log T_t$, but this gives
insufficient information about the domain. Therefore, we use an ad-hoc approach which is sufficient for our purpose. We define the operator in terms of the solution map $(T_t)_{t\ge 0}$ to the parabolic equation. Strictly speaking, $(T_t)_{t\ge 0}$ does not define a semigroup, since due to the presence of the time-dependent weights it does not act on a fixed Banach space. But we simply ignore that and are still able to use standard arguments for semigroups on Banach spaces to identify a dense subset of the domain (compare the discussion below to~\cite[Proposition~1.1.5]{EthierKurtz1986}). However, in that way we do not learn anything about the spectrum of $\mH$. In finite volume, $(T_t)_{t\ge 0}$ is a strongly continuous semigroup of compact operators and we can simply define $\mH$ as its infinitesimal generator. It seems that this would be equivalent to the construction of \cite{Allez2015} through the resolvent equation.

We first discuss the case $d=1$. Then $\xi \in \mC^{\alpha{-}2}(\RR, p(a))$
for all $a > 0$ by assumption, where $\alpha \in (1,\frac32)$. In particular,
$\mH u = (\Delta{+}\xi)u$ is well defined for all $u \in
\mC^\vartheta(\RR,e(l))$ with $\vartheta > 2{-}\alpha$ and $l \in \RR$, and $\mH
u \in \mC^{\alpha{-}2}(\RR,e(l)p(a))$. Our aim is to identify a subset of
$\mC^\vartheta(\RR,e(l))$ on which $\mH u$ is even a continuous function. We can
do this by defining for $t > 0$
\[
	A_t u = \int_0^t T_s u \ud s.
\]
Then $A_t u \in \mC^\vartheta(\RR, e(l{+}t))$, and by definition
\[
	\mH A_t u = \int_0^t \mH T_s u \ud s = \int_0^t \partial_s T_s u \ud s =
	T_t u {-} u \in \mC^\vartheta(\RR, e(l+t)).
\]
Moreover, the following convergence holds in $\mC^\vartheta(\RR, e(l{+}t{+}\varepsilon))$ for
all $\varepsilon > 0$:
\[
	\lim_{n \to \infty} n (T_{1/n} {-} \operatorname{id}) A_t u = \lim_{n
	\to \infty} n \bigg( \int_t^{t+1/n} T_s u \ud s - \int_0^{1/n} T_s u \ud s
	\bigg) = \mH A_t u.
\]
Therefore, we define
\[
	\mD_\mH = \{ A_t u: u \in \mC^\vartheta(\RR, e(l)), l \in \RR, t \in [0,T] \}.
\]
Since for $u \in \mC^\vartheta(\RR, e(l))$ the map $(t \mapsto T_t u)_{t \in
[0,\varepsilon]}$ is continuous in the space $\mC^\vartheta(\RR,
e(l{+}\varepsilon))$ we can find for all $u \in \mC^\vartheta(\RR, e(l))$ a
sequence $\{u^m\}_{m \in \NN} \subset \mD_\mH$ such that $\|u^m {-}
u\|_{\mC^\vartheta(\RR, e(l+\varepsilon))} \to 0$ for all $\varepsilon > 0$.
Indeed, it suffices to set $u^m = m^{-1} A_{m^{-1}} u$. The same construction
also works for $\mH^n$ instead of $\mH$.

In the two-dimensional case $(\Delta {+} \xi)u$ would be well defined whenever
$u \in \mC^\beta(\RR^2,e(l))$ with $\beta > 2{-}\alpha$ for $\alpha \in
(\frac23, 1)$. But in this space it seems impossible to find a domain that is
mapped to continuous functions. And also $(\Delta {+} \xi)u$ is not the right
object to look at, we have to take the renormalization into account and should
think of $\mH = \Delta {+} \xi {-} \infty$. So we first need an appropriate
notion of paracontrolled distributions $u$ for which can define $\mH u$ as a
distribution. As in Proposition~\ref{prop:convergence_PAM} we let $\vartheta \in
(2 {-} 2\alpha, \alpha)$.
\begin{definition}
	Consider $X = (-\Delta)^{-1}\chi(D) \xi$ and $X \diamond \xi$ defined as in
	Assumption~\ref{assu:renormalisation}. We say that $u$ (resp. $u^n$) is \emph{paracontrolled} if $u \in
	\mC^\vartheta(\RR^2, e(l))$ for some $l \in \RR$, and
	\[
		u^\sharp = u {-} u\para X \in \mC^{\alpha + \vartheta}(\RR^2,
		e(l)).
	\]
	 Then set
	\[
		\mH u = \Delta u + \xi \para u + u \para \xi + u^\sharp \reso
		\xi + C_1(u,X,\xi) + u (X\diamond \xi),
	\]
	where $C_1$ is defined in
	Lemma~\ref{lem:weighted_paraproduct_estimates}. The same lemma also
	shows that $\mH u$ is a well defined distribution in
	$\mC^{\alpha{-}2}(\RR^2, e(l) p(a))$.
\end{definition}

The operator $T_t$ leaves the space of paracontrolled distributions invariant, and therefore the same arguments as in $d=1$ give us a domain $\mD_\mH$ such that for all paracontrolled $u$ there exists a sequence $\{u^m\}_{m \in \NN} \subset
\mD_\mH$ with $\|u^m {-} u\|_{\mC^\vartheta(\RR^2, e(l{+}\varepsilon))} \to 0$
for all $\varepsilon > 0$. For general $u \in \mC^\vartheta(\RR^2, e(l))$ and
$\varepsilon > 0$ we can find a paracontrolled $v \in \mC^\vartheta(\RR^2,
e(l))$ with $\| u {-} v \|_{\mC^\vartheta(\RR^2,
e(l{+}\varepsilon))}<\varepsilon$, because $T_t u$ is paracontrolled for all $t
> 0$ and converges to $u$ in $\mC^\vartheta(\RR^2, e(l{+}\varepsilon))$ as $t
\to 0$. Thus, we have established the following result:

\begin{lemma}\label{lem:Anderson-hamiltonian}
	Under Assumption~\ref{assu:renormalisation} let $\vartheta$ be as in
	Proposition~\ref{prop:convergence_PAM}. There exists a domain $\mD_\mH
	\subset \bigcup_{l \in \RR} \mC^\vartheta(\RR^d, e(l))$ such that $\mH u
	= \lim_n n(T_{1/n} {-} \operatorname{id})u$ in $\mC^\vartheta(\RR^d,
	e(l{+}\varepsilon))$ for all $u \in \mD_\mH\cap  \mC^\vartheta(\RR^d,
	e(l))$ and $\varepsilon > 0$ and such that for all $u \in
	\mC^\vartheta(\RR^d, e(l))$ there is  a sequence $\{u^m\}_{m \in
	\NN} \subset \mD_\mH$ with $\|u^m {-} u\|_{\mC^\vartheta(\RR^2,
	e(l{+}\varepsilon))} \to 0$ for all $\varepsilon > 0$. The same is true
	for the discrete operator $\mH^n$ (with $\RR^d$ replaced by $\znd$).
\end{lemma}

\section{The Rough Super-Brownian Motion}\label{sectn:rSBM}
\subsection{Scaling Limit of Branching Random Walks in Random Environment}\label{sectn:convergence}

In this section we consider a deterministic environment, that is a sequence
$\{\xi^n\}_{n \in \NN}$ satisfying Assumption \ref{assu:renormalisation}, to
which we associate the Markov process $\mu^n$ as in Definition
\ref{def:of_u_n_and_mu_n}: Our aim is to prove that the sequence $\mu^n$ converges weakly, with a limit depending on the value of $\vr$. First, we prove tightness  for the
sequence $\mu^n$ in $\DD([0, T] ; \mM(\RR^d))$ for $\vr \ge d/2$. Then, we
prove uniqueness in law of the limit points and thus deduce the weak convergence
of the sequence. Recall that for $\mu \in \mM(\RR^d)$ and $\varphi \in
C_b(\RR^d)$ we use both the notation $\langle \mu, \varphi \rangle$ and
$\mu(\varphi)$ for the integration of $\varphi$ against the measure $\mu$.

\begin{remark}\label{rem:mild_martingale_problem_discrete} 
  Fix $t>0$. For any $\varphi \in L^{\infty}(\znd; e(l))$, for some $l \in \RR$:
  \begin{equation}\label{eqn:mild-martingale-pb} 
	  [0, t] \ni s \mapsto M^{n, \varphi}_t(s) = \mu^n_s (T^n_{t{-}s}\varphi){-}T^n_t\varphi(0) 
  \end{equation}
  is a centered martingale on $[0, t]$ with predictable quadratic variation 
  \[ \langle M^{n,
  \varphi}_t \rangle_s = \int_0^s   \mu^n_r \big( n^{-\vr} |\nabla^n
T^n_{t{-}r}\varphi|^2 + n^{-\vr}|\xi^n_e| (T^n_{t{-}r}\varphi)^2 \big) \ud r.  \]  
\end{remark}

\begin{proof}[Sketch of proof]
Consider a time-dependent function $\psi$. We use Dynkin's formula and an approximation argument applied to the function $(s, \mu)\mapsto F_\psi^t (s, \mu^n) = \mu^n (\psi(s))$: By truncating $F^t_\psi$ and discretizing time and then passing to the limit, we obtain for suitable $\psi$ that
\[
	\mu^n_s(\psi(s)) {-} \mu^n_0(\psi(0)) {-} \int_0^s
	\mu^n_r(\partial_r \psi(r) {+} \mH^n \psi(r)) \ud r
\]
is a martingale with the correct quadratic variation. Now it suffices to note that for $r \in [0, t]: \ \partial_r T^n_{t-r} \varphi = - \mH^n T_{t-r}^n \varphi$.
\end{proof}

For the remainder of this section we assume that $\vr \ge d/2$. To prove the
tightness of the measure-valued process we use the following auxiliary result,
which gives the tightness of the real-valued processes $\{ t\mapsto
\mu^n_t(\varphi)\}_{n \in \NN}$.

The main difficulty in the proof lies in handling the irregularity of the spatial environment. For this reason we consider first the martingale $[0,t]\ni s\mapsto \mu^n_s(T^n_{t-s}\varphi)$ (cf. \eqref{eqn:mild-martingale-pb}) instead of the more natural process $s\mapsto \mu_s^n(\varphi)$. We then exploit the martingale to prove tightness for $\mu^n(\varphi)$. Here we cannot apply the classical Kolmogorov continuity test, since we are considering a pure jump process. Instead we will use a slight variation, due to Chentsov \cite{Chentsov1956Weak} and conveniently exposed in \cite[Theorem 3.8.8]{EthierKurtz1986}.

\begin{lemma}\label{lem:one_dimensional_tightness} For any $l \in \RR$
  and $\varphi \in C^{\infty}(\RR^d, e(l))$ the processes $\{ t\mapsto
  \mu^n(t)(\varphi)\}_{n \in \NN}$ form a tight sequence in $\DD([0,+\infty); \RR).$
\end{lemma}

\begin{proof} 

It is sufficient to prove that for arbitrary $T>0$ the given sequence is tight
in $\DD([0,T]; \RR)$. Hence fix $T>0$ and consider \(0 < \vartheta < 1\) as in
Proposition~\ref{prop:convergence_PAM}. In the following computation $k \in \RR$
may change from line to line, but it is uniformly bounded for $l \in \RR$ and
$T>0$ varying in a bounded set.

\textit{Step 1.} Here the aim is to establish a second moment bound for the increment of the process. Let
$(\mF^n_t)_{t\ge 0}$ be the filtration generated by $\mu^n$. We will prove that the following conditional expectation can be estimated uniformly over $0 \le t \le t{+}h \le T$:
\begin{equation}\label{eqn:one_dimensional_tightness-pr1}
\EE \big[ |\mu^n_{t+h}(\varphi) {-} \mu^n_t(\varphi)|^2 \vert
	\mF^n_t \big] \lesssim h^{\vartheta} \Big[ \mu^n_t(e^{k|x|^{\sigma}})  + |\mu^n_t(e^{k|x|^{\sigma}})|^2 \Big],
\end{equation}
In fact, via the martingales defined in \eqref{eqn:mild-martingale-pb}, one can start by observing that:
\begin{align*} 
	\EE \big[ |\mu^n_{t+h} & (\varphi) {-} \mu^n_t(\varphi)|^2 \vert
	\mF^n_t \big]  \\ 
	& =  \EE \big[ |M^{n, \varphi}_{t{+}h}(t{+}h) {-}
	M^{n,\varphi}_{t{+}h}(t) {+} \mu^n_t(T^n_{h}\varphi {-} \varphi)|^2
      \vert \mF^n_t \big] \\
	&  \lesssim_{\varphi, T} \EE \bigg[ \int_t^{t{+}h} 
	\mu^n_r \big(n^{-\vr} |\nabla^n T^n_{t{+}h{-}r}\varphi|^2 +
      n^{-\vr}|\xi^n_e|(T^n_{t{+}h{-}r}\varphi)^2 \big) \ud r \ \bigg\vert \mF^n_t
    \bigg] \\ 
	& \qquad \qquad \qquad \qquad \qquad \qquad \qquad \qquad \qquad \qquad
	\ \ \  + h^{\vartheta}|\mu^n_t(e^{k|x|^{\sigma}})|^2, 
\end{align*}
where the last term appears since $h\mapsto T^n_h \varphi \in \mL^\vartheta (\znd, e(k))$. 
The first term on the right hand side can be bounded for any $\ve>0$ by:
\begin{equation}\label{eqn:proof_tightness_intermediate_bound}
\begin{aligned}
	\int_t^{t{+}h} & \mu^n_t\bigg( T^n_{r{-}t} \big( n^{-\vr}
	|\nabla^n T^n_{t{+}h{-}r}\varphi|^2 + n^{-\vr}|\xi^n_e|
    (T^n_{t{+}h{-}r}\varphi)^2 \big) \bigg)\ud r \\ 
	& \lesssim \int_t^{t{+}h} \mu^n_t\big(e^{k|x|^{\sigma}} +
	(r{-}t)^{{-}2\varepsilon}e^{k|x|^{\sigma}}\big)\ud r.
\end{aligned}
\end{equation}
Here we have used Lemma \ref{lem:restriction_to_lattice} to ensure that
$\varphi \vert_{\znd}$ is smooth on the lattice together with the a-priori
bound \eqref{eqn:prop_convergence_pam_a_priori_bound} of
Proposition~\ref{prop:convergence_PAM} and with Lemmata~\ref{lem:conv_to_zero}
and \ref{lem:discrete_derivative}, which show respectively a gain of regularity
via the factor $n^{-\vr}$ and a loss of regularity via the discrete derivative
$\nabla^n$, to obtain:
\[ \lim_{n \to \infty} \sup_{r \in [0,T]}\|n^{-\vr} |\nabla^n
T^n_{r}\varphi|^2 \|_{\mC^{\tilde{\vartheta}}(\znd, e(2(l+r)))} = 0,\]
for $0 < \tilde \vartheta < \vartheta {-}1{+}\vr/2$ (we can choose $\vartheta$
sufficiently large so that the latter quantity is positive). Since $\vartheta >0$ and the term is positive, one has by comparison:
$$T^n_{r{-}t} \big( n^{-\vr} |\nabla^n T^n_{t{+}h{-}r}\varphi|^2 \big) \lesssim e^{k |x|^\sigma} \| n^{-\vr} |\nabla^n T^n_{t{+}h{-}r}\varphi|^2 \|_{\mC^{\tilde{\vartheta}}(\znd, e(2(l+T)))}.$$
Moreover, according to Assumption \ref{assu:renormalisation} for $\vr \ge d/2$ the term
$n^{-\vr}|\xi^n_e|$ is bounded in $\mC^{{-}\ve}(\znd, p(a))$ whenever $\ve>0.$ It then follows from the uniform bounds \eqref{eqn:prop_convergence_pam_a_priori_bound} from Proposition~\ref{prop:convergence_PAM} and by applying \eqref{eqn:loose_time_space_regularity} from Lemma~\ref{lem:weighted_paraproduct_estimates}, together with similar arguments to the ones just presented, that:
\begin{align*}
\sup_{n \in \NN}\sup_{r \in [0,T]} &  \|s \mapsto T^n_{s}(n^{-\vr} |\xi^n_e| (T^n_{r}\varphi)^2) \|_{\mM^{2\ve} \mC^{\varepsilon}(\znd, e(k))}\\
& \lesssim  \sup_{n \in \NN}\sup_{r \in [0,T]} \|s \mapsto T^n_{s}(n^{-\vr} |\xi^n_e| (T^n_{r}\varphi)^2) \|_{\mL^{\frac{\vartheta+\ve}{2} +\ve , \vartheta}(\znd, e(k))} \\
& \lesssim  \sup_{n \in \NN}\sup_{r \in [0,T]} \| n^{-\vr} |\xi^n_e| (T^n_{r}\varphi)^2 \|_{\mC^{-\ve}(\znd, e(k))}< \infty.
\end{align*}
This completes the explanation of \eqref{eqn:proof_tightness_intermediate_bound}. So overall, integrating over $r$ we can bound the conditional expectation by:
\begin{equation*}
	h^{1{-}2\ve}\mu^n_t(e^{k|x|^{\sigma}})  +
	h^{\vartheta}|\mu^n_t(e^{k|x|^{\sigma}})|^2 \leq h^\vartheta \Big[ \mu^n_t(e^{k|x|^{\sigma}}) + |\mu^n_t(e^{k|x|^{\sigma}})|^2 \Big],
\end{equation*}
assuming $1-2\ve \geq \vartheta$. This completes the proof of
\eqref{eqn:one_dimensional_tightness-pr1}.

\textit{Step 2.} Now we are ready to apply Chentsov's criterion \cite[Theorem 3.8.8]{EthierKurtz1986}. We have to multiply two
increments of $\mu^n(\varphi)$ on $[t{-}h,h]$ and on $[t,t{+}h]$ and show that for some $\kappa>0$:
\begin{equation}\label{eqn:proof_tightness_final}
\EE \big[ (|\mu^n_{t+h} (\varphi) {-} \mu^n_t(\varphi)| \wedge
	1)^2(|\mu^n_t(\varphi){-}\mu^n_{t{-}h}(\varphi)| \wedge 1)^2 \big] \lesssim h^{1 + \kappa}
\end{equation}
We use \eqref{eqn:one_dimensional_tightness-pr1} to bound:
\begin{align*}\nonumber
	\EE \big[ & (|\mu^n_{t+h} (\varphi) {-} \mu^n_t (\varphi)| \wedge
	1)^2(|\mu^n_t(\varphi){-}\mu^n_{t-h}(\varphi)| \wedge 1)^2 \big] \\
	\nonumber
	& \leq \EE \big[ |\mu^n_{t+h}(\varphi) {-} \mu^n_t(\varphi)|^2
	|\mu^n_t(\varphi){-}\mu^n_{t-h}(\varphi)| \big]\\ 
	\nonumber
	& \lesssim h^\vartheta \EE \Big[ \big(\mu^n_t (e^{k|x|^{\sigma}})  + |\mu^n_t(e^{k|x|^{\sigma}})|^2
    \big)|\mu^n_t (\varphi){-}\mu^n_{t-h}(\varphi)| \Big]\\
	&\lesssim h^\vartheta \EE \Big[ \big(1 + |\mu^n_t(e^{k|x|^{\sigma}})|^2
    \big)|\mu^n_t (\varphi){-}\mu^n_{t-h}(\varphi)| \Big]. 
\end{align*}
By the Cauchy-Schwarz inequality together
with~\eqref{eqn:one_dimensional_tightness-pr1} and the moment bound for
$|\mu^n_t(e^{k|x|^\sigma})|^4$ from Lemma~\ref{lem:moments_estimate} one obtains:
\begin{align*}
	\EE & \bigg[(1+|\mu^n_t (e^{k|x|^{\sigma}})|^2)
	|\mu^n_t(\varphi){-}\mu^n_{t{-}h}(\varphi)| \bigg]\\ 
	& \lesssim \Big( 1+ \EE \big[ |\mu^n_t(e^{k|x|^{\sigma}})|^4 \big]^{1/2} \Big) \EE
	\big[|\mu^n_t(\varphi){-}\mu^n_{t{-}h}(\varphi)|^2 \big]^{1/2} \lesssim  h^{\vartheta/2}.
\end{align*}
Combining all the estimates one finds: 
\[
	\EE \big[ (|\mu^n_{t+h} (\varphi) {-} \mu^n_t(\varphi)| \wedge
	1)^2(|\mu^n_t(\varphi){-}\mu^n_{t-h}(\varphi)| \wedge 1)^2 \big]
	\lesssim h^{\frac 3 2 \vartheta}.
\]
Since $\vartheta > \frac23$, this proves
Equation~\eqref{eqn:proof_tightness_final} for some $\kappa>0$. In particular,
we can apply \cite[Theorem~3.8.8]{EthierKurtz1986} with $\beta = 4$, which
in turn implies that the tightness criterion of Theorem 3.8.6 (b) of the
same book is satisfied. This concludes the proof of tightness for $\{ t\mapsto
\mu^n(t)(\varphi)\}_{n \in \NN}$.
\end{proof}

Consequently, we find tightness of the process $\mu^n$ in the space of
measures.

\begin{corollary}\label{cor:tightness_of_mu_n}
  The processes $\{t \mapsto \mu^n(t)\}_{n \in \NN}$ form a tight sequence in
  $\DD([0, \infty); \mM(\RR^d)).$ 
\end{corollary}

\begin{proof}

  We apply Jakubowski's criterion \cite[Theorem
  3.6.4]{DawsonMaisonneuve1993SaintFlour}. We first need
to verify the compact containment condition. For that purpose note that for all $R > 0$ the set
\(K_R = \{ \mu \in \mM(\RR^d) \ \vert \ \mu(| \cdot|^{2}) \le R \}\) is
compact in $\mM(\RR^d)$. Here \(\mu( | \cdot |^{2}) = \int_{\RR^{d}}
|x|^{2} \ud \mu(x)\). Since the sequence of processes
$\{ \mu^n( | \cdot |^{2})\}_{n \in \NN}$ are tight by
Lemma~\ref{lem:one_dimensional_tightness}, we find for all $T,\varepsilon > 0$
an $R(\varepsilon)$ such that
\begin{align*}
  \sup_n \PP\bigg( \sup_{t \in [0, T]} & \mu^n(t)( | \cdot |^{2} )  \ge
	R(\varepsilon) \bigg) \le \varepsilon,
\end{align*}
as required. Second we note that $C^{\infty}_c(\RR^d)$ is closed under addition and
the maps $\mu \mapsto \{\mu(\varphi)\}_{\varphi \in C^\infty_c(\RR^d)}$ separate
points in $\mM(\RR^d)$. Since Lemma~\ref{lem:one_dimensional_tightness} shows
that $t \mapsto \mu^n(t)(\varphi)$ is tight for any $\varphi \in
C_c^{\infty}(\RR^d)$, we can conclude.
\end{proof}

Next we show that any limit point is a solution to a martingale problem.

\begin{lemma}\label{lem:convergence_martingale_problem}
	Any limit point of the sequence $\{t \mapsto \mu^n(t)\}_{n \in \NN}$ is
	supported in the space of continuous function $C([0, +\infty);
	\mM(\RR^d))$, and it satisfies Property
	$(ii)$ of Definition \ref{def:rSBM} with
	\(\kappa = 0\) if \(\vr > d/2\), and \( \kappa = 2\nu\) if \(\vr = d/2\).
\end{lemma}

\begin{proof}

First, we address the continuity of an arbitrary limit point $\mu$. Since
  $\mM(\TT^d)$ is endowed with the weak topology, it is sufficient to prove the
  continuity of $t\mapsto \langle \mu(t), \varphi \rangle$ for all $\varphi \in
  C_b(\RR^d)$. In view of Corollary~\ref{cor:tightness_of_mu_n}, up to a
  subsequence:
\begin{align*}
  \langle \mu^n, \varphi \rangle \to \langle \mu, \varphi\rangle \ \
  \mathrm{in} \ \ \DD ( [0, \infty); \RR).
\end{align*}
Then by \cite[Theorem 3.10.2]{EthierKurtz1986} in order to obtain the continuity of
  the limit point it is sufficient to observe that the maximal jump size is
  vanishing in $n$:
\begin{align*}
	\sup_{t \geq 0 } |\langle \mu^n_{t}, \varphi \rangle - \langle
	\mu^n_{t-}, \varphi \rangle| \lesssim n^{-\varrho} \| \varphi
	\|_{L^\infty}.
\end{align*}
Next, we study the limiting martingale problem. First we will prove that the
process $M^{\varphi_0, f}_t$ from Definition~\ref{def:rSBM} is a martingale.
Then we will compute its quadratic variation.

\textit{Step 1.}
We fix a limit point $\mu$ and study the required martingale property.	For
$f, \varphi_0$ as required, observe that $\varphi_0^n = \varphi_0 \vert_{\znd}$ is
uniformly bounded in $\mC^{\zeta_0}(\znd; e(l))$ for any $\zeta_0>0$ and $l \in
\RR$, and similarly $f^n = f\vert_{\znd}$ is uniformly bounded in $C([0, t];
\mC^{\zeta}(\znd))$, with an application of Lemma
\ref{lem:restriction_to_lattice}. Hence by Proposition
\ref{prop:convergence_PAM} the solutions $\varphi^n_t$ to the discrete equations \[ \partial_s
\varphi^n_t {+} \mH^n \varphi^n_t = f^n, \qquad \varphi^n_t(t) = \varphi^n_0 \]
converge in $\mL^\vartheta(\RR^d, e(l))$ to $\varphi_t$, up to choosing a
possibly larger $l$. At the discrete level we find, analogously to \eqref{eqn:mild-martingale-pb}, that 
\[ 
  M_t^{\varphi_0, f, n}(s) : = \langle \mu^n(s) , \varphi^n_t(s) \rangle 
  - \langle \mu^{n}(0), \varphi^{n}_{t}(0) \rangle + \int_0^s \ud r \ \langle
  \mu^n(r), f^n(r) \rangle,
\] 
  for \(s \in [0, t]\) is a centered square-integrable martingale. Moreover this martingale is bounded in $L^2$
uniformly over $n$, since the second moment can be bounded via the initial
  value
and the predictable quadratic variation by 
\begin{align*} 
  \EE \Big[ |M^{\varphi_0, f, n}_t|^2(s) \Big] \lesssim
  \int_0^t \! \! \! \ud r \ T^n_{r} \big(
n^{-\vr} |\nabla^n \varphi^n_t(r)|^2 {+} n^{-\vr}|\xi^n| (\varphi^n_t(r))^2 \big)
\end{align*}
and the latter quantity is uniformly bounded in $n$. To conclude that
\(M^{\varphi_{0}, f}_{t}\) is an \(\mF-\)martingale note that
by assumption $M^{\varphi_0, f, n}_{t}$ converges to the continuous process
$M^{\varphi_0, f}_t$. From \cite[Theorem~3.7.8]{EthierKurtz1986} we obtain that for
$0 \le s \le r \le t$ and for bounded and continuous $\Phi\colon \DD([0,s];\mM)
\to \RR$
  \begin{align*}
	      \EE[\Phi(\mu|_{[0,s]}) & (M^{\varphi_0, f}_t(r) - M^{\varphi^0,f}_t(s))] \\
	      &= \lim_n \EE[\Phi(\mu^n|_{[0,s]}) (M^{\varphi_0,
	      f,n}_t(r) - M^{\varphi^0,f,n}_t(s))]=0
  \end{align*}
  by the martingale property. From here we easily deduce
  the martingale property of $M^{\varphi_0, f}_t$.

  \textit{Step 2.} We show that $M^{\varphi_0, f}_t$ has the correct quadratic variation, which should be given as the limit of 
\[
	  \langle M^{\varphi_0, f, n}_t \rangle_s = \int_0^s \ud r \
	\mu^n(r) \big( n^{-\vr} |\nabla^n \varphi^n_t(r)|^2 +
	n^{-\vr}|\xi^n| (\varphi^n_t(r))^2 \big).
\] 
We only treat the case $\vr = d/2$, the case $\vr > d/2$ is similar but
easier because then we can use Lemma~\ref{lem:conv_to_zero} to gain
some regularity from the factor $n^{d/2{-}\vr}$, so that $\|
n^{-\vr}|\xi^n|\|_{\mC^{\ve}(\znd, p(a))} \to 0$ for some $\ve > 0$
and for all $a>0$. 


First we assume, leaving the proof for later, that for any sequence
$\{\psi^n \}_{n \in \NN}$ with $\lim_n \|\psi^n\|_{\mC^{-\ve}(\RR^d, p(a))}
= 0$ for some $a >0$ and all $\ve>0$: 
\begin{equation}\label{eqn:support_eqn_for_QV_convergence}
	\EE\bigg[ \sup_{s\le t} \bigg|\int_0^s \ud r \ \mu^n(r) \big( \psi^n
	\cdot (\varphi^n_t(r))^2 \big)\bigg|^2\bigg]  \longrightarrow 0.
\end{equation}
By Assumption~\ref{assu:renormalisation} we can apply this to $\psi^n = n^{-\vr}|\xi^n| {-} 2\nu$, and
deduce that along a subsequence we have the following weak convergence in $\DD([0,t]; \RR)$:
\[\big(M^{\varphi_0, f, n}_t \big)^2_{\cdot} - \langle M^{\varphi_0, f, n}_t
\rangle_{\cdot}  \longrightarrow \big(M^{\varphi_0, f}_t \big)^2_{\cdot} - \int_0^{ \cdot}
\ud r \ \mu(r)\big( 2\nu (\varphi_t)^2 (r) \big).\]
Note also that the limit lies in $C([0,t]; \RR)$.
If the martingales on the left-hand side are uniformly bounded in $L^2$ we can
deduce as before that the limit is a continuous $L^2-$martingale, and conclude that
\[
\langle M^{\varphi_0, f}_t \rangle_s = \int_0^s \ud r \ \mu(r)\big( 2\nu
(\varphi_t)^2 (r) \big).\] As for the uniform bound in $L^2$, note that it
follows from Lemma \ref{lem:moments_estimate} that
\[
  \sup_n \sup_{0 \le s \le t}\EE \big[ |M^{\varphi_0, f, n}_t (s)|^4\big] <{+}
  \infty. 
\]
For the quadratic variation term we estimate: 
\begin{align*} 
  \EE \big[ \vert
  \langle M^{\varphi_0, f, n}_t \rangle_s \vert^2 \big] \le s \int_0^s \ud
  r \  \EE \big[ \big\vert \mu^n(r)\big( n^{-\vr} |\nabla^n \varphi^n_t(r)|^2 +
      n^{-\vr}|\xi^n| (\varphi^n_t(r))^2 \big)\big\vert^2 \big], 
    \end{align*}
which can be bounded via the second estimate of Lemma
\ref{lem:moments_estimate}. 

\textit{Step 3.} Thus, we are left with the convergence in
\eqref{eqn:support_eqn_for_QV_convergence}. By introducing the martingale from
Equation~\eqref{eqn:mild-martingale-pb} we find that 
\begin{equation}\label{eqn:step-3-support-first}
\begin{aligned}
  \EE \big[ & |\mu^n(r) \big( \psi^n (\varphi^n_t(r))^2 \big)|^2\big] \\
    & \lesssim |T^n_r \big[ \psi^n (\varphi^n_t(r))^2 \big] |^2(0) + \int_0^r \ud q \ T^n_q \Big[ n^{-\vr} \big|\nabla^n
    \big[T^n_{r{-}q}[\psi^n (\varphi^n_t(r))^2] \big]
  \big|^2 \\
  & \qquad \qquad \qquad \qquad \qquad \qquad +  n^{-\vr}|\xi^n| (T^n_{r{-}q}[\psi^n (\varphi^n_t(r))^2])^2 \Big](0).
\end{aligned}
\end{equation}
We start with the first term. By Proposition~\ref{prop:convergence_PAM} we know
that for all $\ve>0$ and $0<\vartheta<1$ satisfying \(\vt + 3 \ve < 1\) and for
\(l>0\) sufficiently large:
\begin{equation}\label{eqn:step-3-support}
  \begin{aligned}
  \| r\mapsto T_r^n [\psi^n (\varphi^n_t(r))^2] & \|_{\mL^{\frac{\vartheta+\ve}{2}
  + \ve, \vartheta}(\znd; e(3l))} \\
  & \lesssim \|\psi^n\|_{\mC^{{-}\ve}(\znd;
  p(a))} \| \varphi^{n} \|_{\mL^{\vt}(\znd; e(l))}^{2}\\
  & \lesssim \|\psi^n\|_{\mC^{{-}\ve}(\znd; p(a))}.
\end{aligned}
\end{equation}
Together with Equation~\eqref{eqn:loose_time_space_regularity} from Lemma~\ref{lem:weighted_paraproduct_estimates} and
\eqref{eqn:step-3-support}, we thus bound: 
\begin{align*} 
|T^n_r \big[ \psi^n (\varphi^n_t(r))^2 \big] |^2(0) & \lesssim r^{-4\ve} \| r\mapsto |T^n_r \big[ \psi^n (\varphi^n_t(r))^2 \big]\|^2_{\mL^{2\ve, \ve}(\znd; e(l))}\\
& \lesssim  r^{{-} 4 \ve} \|
\psi^{n} \|_{\mC^{{-} \ve}(\znd ; p(a))}^{2}.
\end{align*} 
Now we can treat the first term in the integral in
\eqref{eqn:step-3-support-first}. We can choose \(0<\vt<1\) and
\(\ve>0\) with \(\vt + 3 \ve < 1\) such that \(0 < \tilde{\vt} =\vt -1  + d/4\). We then apply
Lemmata \ref{lem:conv_to_zero} and \ref{lem:discrete_derivative}, which guarantee us respectively a regularity gain from the factor $n^{-\frac d 4}$ and a regularity loss from the derivative $\nabla^n$, to obtain:
\begin{align*}
	\| | n^{{-}d/4} \nabla^{n} \big[T^n_{r{-}q}[\psi^n (\varphi^n_t(r))^2] \big]
  \big|^2 \|_{\mC^{\tilde{\vt}}( \znd; e(6l))} & \lesssim
  \| T^n_{r{-}q}[\psi^n (\varphi^n_t(r))^2\|_{\mC^{\vt}( \znd;
    e(3l))}^{2} \\
  & \lesssim (r {-} q)^{{-} (\vt {+} 3\ve)} \| \psi^{n} \|_{\mC^{{-} \ve}(
  \znd; p(a))}^{2},
\end{align*}
where the last step follows similarly to
\eqref{eqn:step-3-support}.
Overall we thus obtain the estimate: 
\begin{align*}
  \int_0^r \ud q \ T^n_q & \big( n^{-\vr} \big|\nabla^n \big[T^n_{r{-}q}[\psi^n
  (\varphi^n_t(r))^2] \big] \big|^2 )(0) \\
  &\lesssim \| \psi^{n} \|_{\mC^{-\ve} ( \znd; p(a))}^{2} \int_0^r \ud q \ (r {-}
  q)^{{-} (\vt {+} 3\ve)} \lesssim \| \psi^{n} \|_{\mC^{{-} \ve}(\znd;
  p(a))}^{2}.
\end{align*}
Following the same steps, one
can treat the second term in the integral in \eqref{eqn:step-3-support-first}.
We now use the same parameter \(\ve\) both for the regularity of \(n^{{-} \vr} |\xi^{n}|\) and of
\(\psi^{n}\), in view of Assumption \ref{assu:renormalisation}, and choose
\(\vt, \ve\) as above with the additional constraint \(\vt + 5 \ve < 1\). Then
we can argue as follows:
\begin{align*}
  \| n^{-\vr}|\xi^n| (T^n_{q}[\psi^n (\varphi^n_t(r))^2])^2 \|_{\mC^{ {-}
  \ve}(\znd; e(2l)p(a))} \lesssim q^{{-} (\vt {+}3\ve)} \| \psi^{n}
  \|_{\mC^{{-} \ve}( \znd; p(a))}^{2}
\end{align*}
and hence:
\begin{align*}
  \int_{0}^{r} & \ud q \  T^{n}_{q} ( n^{-\vr}|\xi^n| (T^n_{q}[\psi^n
  (\varphi^n_t(r))^2])^2)(0) \\
  & \lesssim \| \psi^{n} \|_{\mC^{{-} \ve}(\znd; p(a))}^{2} \int_0^r \ud q \ (r
  {-} q)^{{-} (\vt {+} 3\ve)} q^{ {-}2 \ve }  \\
  & \lesssim \|\psi^n\|_{\mC^{{-}\ve}(\znd; p(a))}^{2},
\end{align*}
where in the last step we used that \(\vt+ 5 \ve < 1\). This concludes the
proof.
\end{proof}

Our first main result, the law of large numbers, is now an easy consequence.

\begin{proof}[Proof of Theorem \ref{thm:LLN}]
Recall that now we assume $\vr > d/2.$ In view of Corollary
\ref{cor:tightness_of_mu_n} we can assume that along a subsequence $\mu^{n_k}
\Rightarrow \mu$ in distribution in $\DD([0, +\infty); \mM( \RR^d ))$.
It thus suffices to prove that $\mu = w$. The previous lemma shows that for $\varphi \in C^\infty_c(\RR^d)$ the
process $s \mapsto \mu(s)(T_{t{-}s}\varphi){-} T_{t}\varphi(0)$ is a continuous
square-integrable martingale with vanishing quadratic variation. Hence, it is
constantly zero and $\mu(t)(\varphi) = T_t\varphi(0) = (T_t
\delta_0)(\varphi)$ almost surely for each fixed $t\ge 0$. Note that \(T_{\cdot}
\delta_{0}\) is well-defined, as explained in Remark
\ref{rem:PAM_Green_function}. Since $\mu$ is continuous, the identity holds
almost surely for all $t > 0$. The identity $\mu(t) = T_t \delta_0$ then follows
by choosing a countable separating set of smooth functions in
$C^\infty_c(\RR^d)$.
\end{proof}

Now we pass to the case $\vr = d/2.$ To deduce the weak convergence of the sequence
$\mu^n$ we have to prove that the distribution of the limit points is unique.
For that purpose we first introduce a duality principle for the Laplace
transform of our measure-valued process, for which we have to study Equation
\eqref{eqn:PDE_laplace_duality_continuous}.  We will consider mild solutions,
i.e. $\varphi$ solves~\eqref{eqn:PDE_laplace_duality_continuous} if and only if
\[
\varphi(t) = T_t\varphi_0 - \frac{\kappa}{2} \int_0^t \ud s \ T_{t{-}s}(\varphi(s)^2).\]
We shall denote the solution by $\varphi(t) = U_t \varphi_0$, which is
justified by the following existence and uniqueness result:

\begin{proposition}\label{prop:existence_solutions_duality_PDE}
  Let $T, \kappa>0$, $l_0 < -T$ and $\varphi_0 \in C^{\infty}(\RR^d, e(l_0))$
  with $\varphi_0 \ge 0$. For $l = l_0 + T$ and $\vartheta$ as in Proposition
  \ref{prop:convergence_PAM} there is a unique mild solution $\varphi \in
  \mL^{\vartheta}(\RR^d, e(l))$ to Equation
  \eqref{eqn:PDE_laplace_duality_continuous}: \begin{align*} \partial_t\varphi =
  \mH \varphi {-} \frac{\kappa}{2} \varphi^2, \qquad \varphi(0) = \varphi_0.
\end{align*} We write $U_t\varphi_0 := \varphi(t)$ and we have the following
bounds:
\[
	0 \le U_t\varphi_0 \le T_t \varphi_0, \quad \| \{U_t\varphi_0\}_{t \in
	[0,T]} \|_{\mL^{\vartheta}(\RR^d, e(l))} \lesssim e^{C\|\{T_t
	\varphi_0\}_{t \in [0,T]}\|_{CL^{\infty}(\RR^d, e(l))}}.
\]
\end{proposition}

\begin{proof}
 We define the map
  $\mI(\psi) = \varphi$, where $\varphi$ is the solution to
  \[
	  \partial_t \varphi = \big(\mH {-} \frac\kappa2 \psi\big) \varphi, \qquad
	\varphi(0) = \varphi_0.
  \]
  If $l_0 < -T$, then $(T_t \varphi_0)_{t \in [0,T]} \in
  \mL^{\vartheta}(\RR^d, e(l))$ for $l = l_0 + T$, and thus a slight adaptation
of the arguments for Proposition~\ref{prop:convergence_PAM} shows that
$\mI$ satisfies
  \[
	\mI \colon \mL^{\vartheta}(\RR^d, e(l)) \to \mL^{\vartheta}(\RR^d,
	e(l)), \qquad \|\mI(\psi)\|_{\mL^{\vartheta}(\RR^d, e(l))} \lesssim
	e^{C\|\psi\|_{CL^{\infty}(\RR^d, e(l))}}
  \]
  for some $C>0$. Moreover, for positive $\psi$
this map satisfies the bound \( 0 \le \mI(\psi)(t) \le T_t\varphi_0, \)
so in particular we can bound $\|\mI(\psi)\|_{CL^{\infty}(\RR^d, e(l))} \le \|\{T_t
\varphi_0\}_{t \in [0,T]}\|_{CL^{\infty}(\RR^d, e(l))}$. Now, define $\varphi^0(t,x) =
T_t \varphi_0 (x)$ and then iteratively $\varphi^m =
\mI(\varphi^{m-1})$  for $m \ge 1.$
This means that $\varphi^m$ solves the equation:
\begin{align*}
\partial_t \varphi^m = \mH \varphi - \frac{\kappa}{2} \varphi^{m-1} \varphi^m.
\end{align*}
Hence our a priori bounds guarantee that \[
\sup_m \| \varphi^m \|_{\mL^{\vartheta}(\RR^d, e(l))}  \lesssim
e^{C\|\{T_t \varphi_0\}_{t \in [0,T]}\|_{CL^{\infty}(\RR^d, e(l))}}.  \] By compact embedding of
$\mL^{\vartheta}(\RR^d, e(l)) \subset \mL^{\zeta}(\RR^d, e(l^\prime))$ for
$\zeta < \vartheta$, $l^\prime < l$ we obtain convergence of a subsequence in the
latter space. The regularity ensures that the limit point is indeed a solution
to Equation \eqref{eqn:PDE_laplace_duality_continuous}. The uniqueness of such a
fixed-point follows from the fact that the difference $z=\varphi{-}\psi$ of two solutions
$\varphi$ and $\psi$ solves the well posed linear equation: $\partial_t z = \big(\mH {+}
\frac\kappa2(\varphi{+}\psi)\big) z$ with $z(0) = 0$, and thus $z = 0$.
\end{proof}

We proceed by proving some implications between Properties $(i)-(iii)$ of
Definition~\ref{def:rSBM}.

\begin{lemma}\label{lem:implications_rsbm_definition}
	In Definition \ref{def:rSBM} the following implications hold between the three properties:
	\[ (ii) \Rightarrow (i), \qquad \qquad (ii) \Leftrightarrow (iii).\]
\end{lemma}

\begin{proof}
	$(ii) \Rightarrow (i)$: Consider \(U_{\cdot} \varphi_0\) as in point
	$(i)$ of Definition~\ref{def:rSBM}, which is well defined in view of
	Proposition \ref{prop:existence_solutions_duality_PDE}. An application
	of It\^o's formula and Property (ii) of Definition~\ref{def:rSBM} with $\varphi_t(s) = U_{t-s} \varphi_0$, guarantee that for any
	$F \in C^2(\RR)$, and for $f(r) = \frac\kappa2 (U_{t {-} r} \varphi_0
	)^2$:
\begin{align*} 
  &F( \langle \mu(t), \varphi_0 \rangle) = F(\langle \mu(s), U_{t
  {-} s} \varphi_0 \rangle) {+} \int_s^t \ud r \ F'(\langle \mu(r), U_{t {-} r}
  \varphi _0 \rangle) \langle \mu(r), f(r) \rangle \\ 
  & + \hh \int_s^t
  F''(\langle \mu(r),U_{t {-} r} \varphi_0\rangle) \ud \langle M^{\varphi_0,
f}_t \rangle_r {+} \int_s^t  \ F'(\langle \mu(r), U_{t {-} r} \varphi_0 \rangle)
\ud M^{\varphi_0,f}_t(r) , 
\end{align*} 
where $\ud \langle M^{\varphi_0,f}_t \rangle_r = \langle \mu(r), \kappa
(U_{t-r}\varphi_0)^2 \rangle \ud r = \langle \mu(r), 2
f(r) \rangle \ud r$. We apply this for $F(x) = e^{-x}$, so that $F'' = - F'$ and the two Lebesgue integrals cancel. Since $F'$ is
bounded for positive $x$ the stochastic integral is a true martingale and we deduce property $(i)$.
  
 $(ii) \Rightarrow (iii)$: Let $\varphi \in \mD_{\mH}$ and $t>0$ and let $0 =
 t^n_0 \le t^n_1 \le \ldots \le t^n_{n} = t$, $n \in \NN$, be a
sequence of partitions of $[0, t]$ with $\max_{k \le n-1} \Delta^n_k := \max_{k
\le n-1} (t^{n}_{k{+}1} {-} t^n_k) \to 0$. Then
	  \begin{align*}
		  \langle\mu(t),\varphi\rangle - & \langle\mu(0),\varphi\rangle\\
		  & = \sum_{k=0}^{n-1} \big[\big(\langle\mu(t^n_{k+1}),\varphi\rangle	 {-} \langle\mu(t^n_k),T_{\Delta^n_k} \varphi\rangle\big) +  \langle\mu(t^n_k), T_{\Delta^n_k}\varphi {-} \varphi \rangle \big] \\
		  & = \sum_{k=0}^{n-1} \big[\big( M^{\varphi, 0}_{t^n_{k+1}}(t^n_{k+1}) {-} M^{\varphi, 0}_{t^n_{k+1}}(t^n_k)\big) + \Delta_k^n \langle\mu(t^n_k), \frac{T_{\Delta^n_k}\varphi {-} \varphi}{\Delta^n_k} \rangle \big].
	  \end{align*}
	We start by studying the second term on the right hand side:
	\begin{align*}
		\sum_{k=0}^{n-1} \Delta_k^n \langle\mu(t^n_k), &
		\frac{T_{\Delta^n_k}\varphi {-} \varphi}{\Delta^n_k} \rangle \\
		& =
		\sum_{k=0}^{n-1} \Big[ \Delta_k^n \langle\mu(t^n_k),
		\frac{T_{\Delta^n_k}\varphi {-} \varphi}{\Delta^n_k}  {-} \mH
		\varphi \rangle +  \Delta_k^n \langle\mu(t^n_k),
		\mH \varphi \rangle \Big] \\
		& =: R_n + \sum_{k=0}^{n-1} \Delta_k^n \langle\mu(t^n_k), \mH
		\varphi \rangle.
      \end{align*}
      By continuity of $\mu$ the second term on the right hand side converges
      almost surely to the Riemann integral $\int_0^t \langle \mu(r), \mH
      \varphi \rangle \ud r$. Moreover, from the characterization $(ii)$ we get
	$\EE[\mu(s)(\psi)] = \langle \mu(0), T_s \psi \rangle$ and
	  \[
		\EE[\mu(s)(\mH \varphi)^2] \lesssim \langle\mu(0),
		(T_s(\mH\varphi))^2 \rangle + \int_0^s \ud r \, \langle \mu(0),  T_r \Big[(T_{s{-}r}\mH \varphi)^2 \Big] \rangle,
	  \]
	which is uniformly bounded in $s \in [0,t]$. So the sequence is
	uniformly integrable and converges also in $L^1$ and not just almost
	surely. Moreover,
	\[
		  \EE[|R_n|] \lesssim \sum_{k=0}^{n-1} \Delta^n_k \big\langle
		\mu_0, T_{t^n_k}(| (\Delta^n_k)^{-1}(T_{\Delta^n_k}\varphi {-}
		\varphi)  {-} \mH \varphi |)\big\rangle,
	  \]
	  and since Lemma~\ref{lem:Anderson-hamiltonian} implies that $\max_{k \le n-1} (\Delta^n_k)^{-1}(T_{\Delta^n_k}\varphi {-}
	\varphi)$ converges to $\mH \varphi$ in $\mC^\vartheta(\RR^d, e(l))$ for
	some $l \in \RR$ and $\vartheta > 0$ (so in particular uniformly), it
	follows from Proposition~\ref{prop:convergence_PAM} and the assumption
	$\langle \mu_0, e(l)\rangle < \infty$ for all $l \in \RR$ that
	$\EE[|R_n|] \to 0$. Thus, we showed that
	  \begin{align*}
		  L^\varphi_t &= \langle\mu(t),\varphi\rangle	 -
		\langle\mu(0),\varphi\rangle - \int_0^t \langle \mu(r), \mH
		\varphi\rangle \ud r \\
		& = \lim_{n \to \infty}  \sum_{k=0}^{n-1}
		\big( M^{\varphi, 0}_{t^n_{k+1}}(t^n_{k+1}) {-} M^{\varphi,
		0}_{t^n_{k+1}}(t^n_k)\big),
	  \end{align*}
	  and the convergence is in $L^1$. By taking partitions that contain $s
	\in [0,t)$ and using the martingale property of $M^{\varphi,0}_r$ we get
	$\EE[L^\varphi(t) | \mF_s] = L^\varphi(s)$, i.e. $L^\varphi$ is a
	martingale. By the same arguments that we used to show the uniform
	integrability above, $L^\varphi(t)$ is square integrable for all $t >
	0$. To derive the quadratic variation we use again a sequence of
	partitions containing $s \in [0,t)$ and obtain
	  \begin{align*}
		\EE\big[ L^\varphi(t)^2 {-} & L^\varphi(s)^2 \big|\mF_s \big] =
		\EE\big[ (L^\varphi(t) {-} L^\varphi(s))^2 \big|\mF_s \big] \\
		& = \lim_{n \to \infty}  \sum_{k: t^n_{k+1} > s} \EE\big[\big(
		M^{\varphi, 0}_{t^n_{k+1}}(t^n_{k+1}) {-} M^{\varphi,
	    0}_{t^n_{k+1}}(t^n_k)\big)^2 \big| \mF_s \big] \\
		& = \lim_{n \to \infty}  \sum_{k: t^n_{k+1} > s} \EE\Big[ \kappa
		\int_{t^n_k}^{t^n_{k+1}} \ud r \,\langle \mu(r),
	      (T_{t^n_{k+1}{-}r}\varphi)^2 \rangle \Big| \mF_s \Big] \\
		& = \EE \Big[ \kappa \int_s^t  \ud r \,\langle \mu(r), \varphi^2
		\rangle \Big| \mF_s \Big].
	  \end{align*}
	Since the process $\kappa \int_0^\cdot \ud r \,\langle \mu(r), \varphi^2
	\rangle$ is increasing and predictable, it must be equal to $\langle
	L^\varphi\rangle$.
  
  $(iii) \Rightarrow (ii)$: Let $t \ge 0$, $\varphi_0 \in \mD_{\mH}$, and let $f\colon [0, t] \to \mD_{\mH}$ be a
  piecewise constant function (in time; it might seem more natural to take $f$ continuous, but since we did not equip $\mD_\mH$ with a topology this has no clear meaning). We write $\varphi$ for the solution to
the backward equation
\[
	(\partial_s {+} \mH )\varphi = f, \qquad \varphi(t) = \varphi_0,
\]
which is given by $\varphi(s) = T_{t-s} \varphi_0 + \int_s^{t} T_{r-s} f(r) \ud
r$. Note that by assumption $\varphi(r) \in \mD_{\mH}$ for all $r \le
t$. For $0 \le s \le t$, let $0 = t^n_0 \le t^n_1 \le \ldots \le t^n_{n} = s$, $n \in \NN$, be a
sequence of partitions of $[0, s]$ with $\max_{k \le n-1} \Delta^n_k := \max_{k
\le n-1} (t^{n}_{k{+}1} {-} t^n_k) \to 0$.
Similarly to the computation in the step ``$(ii)\Rightarrow (iii)$'' we can decompose: 
\begin{align*} 
  \langle\mu(s), \varphi(s)\rangle & {-} \langle \mu(0),
  \varphi(0)\rangle = \\ 
  = & \sum_{k = 0}^{n{-}1}  \bigg[ L^{\varphi(t^n_{k{+}1})} (t^n_{k{+}1} ) {-}
  L^{\varphi(t^n_{k{+}1})} ( t^n_{k}) {+} \int_{t^n_k}^{t^n_{k+1}} \ud r \
  \langle \mu(r), f(r)\rangle \bigg] + R_n,
\end{align*}
with
\begin{align*}
	R_n & = \sum_{k = 0}^{n{-}1} \int_{t^n_k}^{t^n_{k{+}1}} \ud r \bigg[  \langle\mu(r),  \mH \varphi(t_{k+1}^n)\rangle {-} \langle \mu(t^n_k), (\Delta_k^n)^{-1} (T_{\Delta^n_k} {-} \operatorname{id})\varphi(t^n_{k+1})\rangle \\
		&\hspace{80pt} + \langle \mu(t^n_k), T_{r-t^n_k} f(r)\rangle {-} \langle \mu(r), f(r) \rangle \bigg].
\end{align*}
By similar arguments as in the step $(ii) \Rightarrow (iii)$ we see that $R_n$
converges to zero in $L^1$, and therefore $s \mapsto \langle \mu(s),
\varphi(s)\rangle {-} \langle \mu(0),\varphi(0)\rangle {-} \int_0^s \ud r \,
\langle \mu(r), f(r)\rangle$ is a martingale. Square integrability and the right
form of the quadratic variation are shown again by similar arguments as before.

By density of $\mD_{\mH}$ it follows that $M^{\varphi_0, f}_t$ is a martingale on
$[0, t]$ with the required quadratic variation for any $\varphi_0 \in
C_c^{\infty}(\RR^d)$ and $f \in C([0,t]; \mC^{\zeta}(\RR^d))$ for $\zeta>0$. This concludes the proof.
\end{proof}

Characterization $(i)$ of Definition~\ref{def:rSBM} enables us to deduce the
uniqueness in law and then to conclude the proof of the equivalence of the
different characterizations in Definition~\ref{def:rSBM}.

\begin{proof}[Proof of Lemma~\ref{lem:rSBM-uniqueness-equivalence}]
  First, we claim that uniqueness in law follows from Property $(i)$ of Definition~\ref{def:rSBM}. Indeed, we have for $0 \le
  s\le t$ and $\varphi \in C^\infty_c(\RR^d)$, $\varphi \ge 0$ that \( \EE \big[
  e^{- \langle  \mu(t), \varphi \rangle } \big\vert \mF_s\big] = e^{-\langle
\mu(s), U_{t {-} s} \varphi \rangle}\).
  For $s = 0$ we can use the Laplace transform and the linearity of $\varphi
    \mapsto \langle  \mu(t), \varphi \rangle$ to deduce that the law of $(
    \langle  \mu(t), \varphi_1 \rangle,\dots,  \langle  \mu(t), \varphi_n
  \rangle)$ is uniquely determined by $(i)$ whenever $\varphi_1,\dots, \varphi_n$
  are positive functions in $C^\infty_c(\RR^d)$. By a 
  monotone class argument (cf. \cite[Lemma
  3.2.5]{DawsonMaisonneuve1993SaintFlour}) the law of $\mu(t)$ is unique. We
  then see
  inductively that the finite-dimensional distributions of $\mu =
  \{\mu(t)\}_{t\ge 0}$ are unique, and thus that the law of $\mu$ is unique.
    
    It remains to show the implication $(i) \Rightarrow (ii)$ to conclude the
    proof of the equivalence of the characterizations in
    Definition~\ref{def:rSBM}. But we showed in
    Lemma~\ref{lem:convergence_martingale_problem} that there exists a process
    satisfying $(ii)$, and in Lemma~\ref{lem:implications_rsbm_definition} we
    showed that then it must also satisfy $(i)$. And since we just saw that
    there is uniqueness in law for processes satisfying $(i)$ and since Property
    $(ii)$ only depends on the law and it holds for one process satisfying
    $(i)$, it must hold for all processes satisfying $(i)$. (Strictly speaking
    Lemma~\ref{lem:convergence_martingale_problem} only gives the existence for
  $\kappa = 2\nu \in (0,1]$, but see Section~\ref{sectn:mixing_dawson_watanabe}
below for general $\kappa$.)
 \end{proof}

Now the convergence of the sequence \(\{\mu^n\}_{n \in \NN}\) is an easy consequence:

\begin{proof}[Proof of Theorem \ref{thm:CLT}]

  This follows from the characterization of the limit points from
	Lemma~\ref{lem:convergence_martingale_problem} together with the
	uniqueness result from Lemma \ref{lem:rSBM-uniqueness-equivalence}.
\end{proof}

\subsection{Mixing with a classical Superprocess}\label{sectn:mixing_dawson_watanabe} 

In Section~\ref{sectn:convergence} we constructed the rSBM of parameter
\(\kappa=2\nu\), for \(\nu\) defined via Assumption \ref{assu:noise} which leads
to the restriction \(\nu \in (0, \frac12]\). This section is devoted to
constructing the rSBM for arbitrary $\kappa >0$. We do so by means of an
interpolation between the rSBM and a Dawson-Watanabe superprocess (cf.
\cite[Chapter 1]{Etheridge2000}). Let $\Psi$ be the generating function of a
discrete finite positive measure $\Psi(s) = \sum_{k \ge 0}p_k s^k$ and
\(\xi^n_p\) a controlled random environment associated to a parameter \(\nu =
\EE [\Phi_+]\). We consider the quenched generator: 
\begin{align*} 
  \mL^{n , \omega^p}_{\psi}(F)(\eta) & = \sum_{x \in \ZZ_n^d} \eta_x  \cdot
  \bigg[ \Delta^{n} F (\eta) + (\xi^n_{p,e})_{+}(\omega^p, x) d^{1}_{x} F(
    \eta) \\
    & \qquad \qquad + (\xi^n_{p,e})_{-}( \omega^p, x) d^{{-} 1}_{x}F(\eta) +
  n^{\vr}\sum_{k \ge 0} p_k d^{ (k {-} 1)}_{x}F(\eta) \bigg] 
\end{align*} 
with the notation \(d^{k}_{x}F( \eta) = F( \eta^{x; k}) {-} F(\eta)\), where for \(k \geq {-} 1\) we write 
$\eta^{x; k }(y) = (\eta(y) {+}  k 1_{\{x\}}(y))_+$.

\begin{assumption}[On the Moment generating function] We assume that
$\Psi'(1) = 1$ (critical branching, i.e. the expected number of offsprings in
one branching/killing event is $1$) and we write $\sigma^2 = \Psi''(1)$ for the
variance of the offspring distribution.
\end{assumption}

Now we introduce the associated process. The construction of the process $\bar{u}^n$ is analogous to the case without
$\Psi$, which is treated in Appendix \ref{sectn:construction_markov_process}.

\begin{definition}\label{def:of_dawson_watanabe_discrete}
  
  Let $\vr \ge d/2$ and let $\Psi$ be a moment generating function satisfying the previous assumptions.
  Consider a controlled random environment $\xi^n_p$ associated to a parameter
  \(\nu \in (0,\frac12]\). Let $\PP^n = \PP^p \ltimes \PP^{n, \omega^p}$ be the measure on
  $\Omega^p \times \DD([0, {+}\infty); E)$ such that for fixed \(\omega^p \in
  \Omega^p\), under the measure \(\PP^{n, \omega^p}\) the canonical process on
  $\DD([0, {+}\infty); E)$ is the Markov process $\bar{u}^n_p(\omega^p, \cdot)$
  started in $\bar{u}^n_p(0) = \lfloor n^\vr  \rfloor 1_{\{ 0\}}(x)$ associated
  to the generator $\mL^{\omega^p, n}_{\Psi}$ defined as above. To $\bar{u}^n_p$
  we associate the measure valued process
	\[ 
		   \langle \bar{\mu}^n_p(\omega^p, t), \varphi \rangle = \sum_{x
		\in  \mathbb{Z}_n^d} \bar{u}^n_p(\omega^p, t,x) \varphi(x)
		\lfloor n^{\vr} \rfloor^{-1}
	   \]
  for any bounded $\varphi\colon \znd \to \RR$. With this definition $\bar{\mu}^n_p$
  takes values in $\Omega^p \times \DD([0, T]; \mM(\RR^d))$ with the law induced
  by $\PP^{n}$.  
  \end{definition}

\begin{remark}
  As in Remark~\ref{rem:mild_martingale_problem_discrete} we see that for
  $\varphi \in L^\infty(\znd,e(l))$ with $l\in \RR$ the process $\bar{M}^{n,
  \varphi}_t(s) :
  = \bar{\mu}^n(s)(T^n_{t{-}s} \varphi) {-} T^n_t\varphi (0)$ is a martingale
  with predictable quadratic variation: 
\[ \langle \bar{M}^{n, \varphi}_t \rangle_s = \int_0^s \ud r \ \bar{\mu}^n(r)
\big(n^{-\vr} |\nabla^n T^n_{t{-}r}\varphi|^2 + (n^{-\vr}|\xi^n_e|{+}  \sigma^2)
(T^n_{t{-}r}\varphi)^2 \big).  \]
\end{remark}

In view of this Remark, we can follow the discussion of Section
\ref{sectn:convergence} to deduce the following result (cf. Corollary
\ref{cor:convergence_total_law}).

\begin{proposition}\label{prop:convergence_mixiture_dawson} 
  
  The sequence of measures \(\PP^n\) as in Definition
  \ref{def:of_dawson_watanabe_discrete} converge weakly as measures on \(
  \Omega^p \times \DD([0, T]; \mM(\RR^d))\) to the measure \(\PP^{p} \times
  \PP^{\omega^p}\) associated to a rSBM of parameter \(\kappa = 1_{\{\vr = \frac{d}{2}\}} 2\nu {+}
  \sigma^2\), in the sense of Theorem \ref{thm:CLT} and Corollary
  \ref{cor:convergence_total_law}. In short, we write \( \overline{\mu}^n_p \to
  \overline{\mu}_p\). 
\end{proposition}

In particular the rSBM is also the scaling limit of critical branching random
walks whose branching rates are perturbed by small random potentials.

\section{Properties of the Rough Super-Brownian Motion}\label{sectn:properties}
\subsection{Scaling Limit as SPDE in d=1}\label{secnt:SPDE_derivation}

In this section we characterize the rSBM in dimension $d
= 1$ as the solution to the SPDE \eqref{eqn:spde_rsbm_formal_with_kappa} in the
sense of Definition \ref{def:solutions_to_SPDE_1D}. For that purpose we first show that the random measure $\mu_p$ admits a density with respect
to the Lebesgue measure.

\begin{lemma}\label{lem:existence_of_density}
	Let $\mu$ be a one-dimensional rSBM of parameter
	$\nu$. For any $\beta < 1/2$, $p \in [1, 2/(\beta{+}1))$ and $l \in
	\RR$, we have:
	\[
		\EE\big[\|\mu\|_{L^p ( [0, T]; B^{\beta}_{2,2}(\RR, e(l)))}^p\big] < \infty.
	\]
\end{lemma}

\begin{proof}
	Let $t>0$ and $\varphi \in C^\infty_c(\RR)$. By Point $(ii)$ of
	Definition~\ref{def:rSBM} the process $M^{\varphi}_t(s) = \langle
	\mu(s), T_{t-s} \varphi \rangle - \langle \mu(0), T_t \varphi \rangle$,
	$s\in [0,t]$, is a continuous square-integrable martingale with
	quadratic variation $\langle M^{\varphi}_t \rangle_s = \int_0^s \langle
	\mu(r), (T_{t-r}\varphi)^2 \rangle$. Using the moment
	estimates of Lemma~\ref{lem:moments_estimate}, which by Fatou's lemma
	also hold for the limit $\mu$ of the $\{\mu^n\}$, this
	martingale property extends to \(\varphi \in
	\mC^{\vt}( \RR, e(k))\) for arbitrary \(k \in \RR\) and \(\vt>0\). In
	particular, for such \(\varphi\) we get
	\[
		\EE[\langle \mu(t), \varphi \rangle^2 ] \lesssim \int_0^t T_r(
		(T_{t-r} \varphi)^2) (0) \ud r + (T_t \varphi)^2(0).
	\]
	Now note that \(\EE\big[ \|\mu(t) \|_{B^\beta_{2,2}(e(l))}^2\big] = \sum_j
	2^{2j\beta} \int \EE[\langle \mu(t), K_j(x-\cdot)\rangle^2]
	e^{{-}2l|x|^\sigma} \ud x \), so we apply this estimate with $\varphi = K_j(\cdot - x)$:
	\begin{align}\label{eqn:existence_of_density-pr1}
	  \EE[\langle \mu(t), K_j(x{-}\cdot)\rangle^2] \lesssim \! \!
	  \int_0^t \! \! T_r( (T_{t-r} K_j(x {-} \cdot))^2) (0)\! \ud r
	  {+} (T_t K_j(x {-} \cdot))^2(0).
	\end{align}
	We start by proving that $\| K_j(x-\cdot)
	\|_{\mC^\alpha_1(\RR, e(k))} \lesssim 2^{j\alpha} e^{{-} k |x|^{\sigma}}
	$ for any $k>0$. Indeed, using that $K_i$ is an even function and writing $\tilde
	K_{i-j} = 2^{(i-j)d}K_0(2^{i-j}\cdot)\ast K_0$ if $i,j \ge 0$ and
	appropriately adapted if $i=-1$ or $j=-1$, we have:
	\begin{align*}
		\| \Delta_i & (K_j(x-\cdot))e(k)\|_{L^1(\RR)} = 1_{\{|i-j|\le
		1\}} \int_{\RR^d} |K_i \ast K_j(x-y)| e^{-k|y|^\sigma} \ud y \\
		& = 1_{\{|i-j|\le 1\}} \int_{\RR} |\tilde K_{i-j}(y)|
		e^{-k|x-2^{-j}y|^\sigma} \ud y \\
		& \lesssim 1_{\{|i-j|\le 1\}} \int_{\RR} |\tilde K_{i-j}(y)|
		e^{k|2^{-j}y|^\sigma-k|x|^\sigma} \ud y  \lesssim 1_{\{|i-j|\le
		1\}} e^{-k|x|^\sigma},
	 \end{align*}
	where in the last step we used that $|\tilde K_{i-j}(y)| \lesssim
	e^{-2k|y|^\sigma}$ and $2^{-j\sigma} \le 2^\sigma < 2$. 
	
	Now, for $\zeta < 0$ satisfying the assumptions of
	Proposition~\ref{prop:convergence_PAM} and for  $p \in [1,\infty]$ and
	sufficiently small  $\varepsilon > 0$:
	\begin{align*}
		\|T_s K_j(x-\cdot) \|_{\mC^\varepsilon_p(\RR,e(k+s))} 
		& \lesssim \|T_s K_j(x-\cdot) \|_{\mC^{1-\frac1p + \varepsilon}_1(\RR,e(k+s))}\\ 
		& \lesssim 2^{j\zeta} s^{(\zeta-1+\frac1p - 2\varepsilon)/2} e^{-k|x|^\sigma}.
	\end{align*}
	To control the first term on the right hand side
	of~\eqref{eqn:existence_of_density-pr1}, we apply this with $p=2$ and
	obtain for $t \in [0,T]$ and $\zeta > - 1/2$
	\begin{align*}
	  \int_0^t  T_r( (T_{t-r} & K_j(x {-} \cdot ))^2) (0) \ud r\\
	  &  \lesssim
	  \int_0^t \|T_r( (T_{t-r} K_j(x {-} \cdot
		))^2)\|_{\mC^\varepsilon_\infty(\RR,e(2k+T))} \ud r \\
		& \lesssim \! \int_0^t \!  \|T_r( (T_{t-r} K_j(x {-} \cdot
		))^2)\|_{\mC^{1+\varepsilon}_1(\RR,e(2k+T))} \! \ud r \\
		& \lesssim \!  \int_0^t \!  r^{-\frac{1+2\varepsilon}{2}} \|(T_{t-r}
		K_j(x {-} \cdot ))^2\|_{\mC^{\varepsilon}_1(\RR,e(2k))}  \ud r \\
		& \lesssim \! \! \int_0^t \! \! r^{-\frac{1+2\varepsilon}{2}} \|T_{t-r}
		K_j(x {-} \cdot )\|^2_{\mC^{\varepsilon}_2(\RR,e(k))} \! \ud r \\
		& \lesssim \! \! \int_0^t \! \! r^{-\frac{1+2\varepsilon}{2}} (2^{j\zeta}
		(t{-}r)^{\frac{\zeta-\frac12 - 2\varepsilon}{2}} e^{-k|x|^\sigma})^2
		\!  \ud r \\
		& \simeq 2^{2j\zeta} e^{-2k|x|^\sigma}
		t^{1-\frac{1+2\varepsilon}{2} + \zeta-\frac12 - 2\varepsilon} =
		2^{2j\zeta} e^{-2k|x|^\sigma}
		t^{\zeta-3\varepsilon}, 
	\end{align*}
	where we used that $\int_0^t r^{-\alpha} (t-r)^{-\beta} \ud r \simeq
	t^{1-\alpha-\beta}$ for $\alpha, \beta < 1$. The second term on the
	right hand side of~\eqref{eqn:existence_of_density-pr1} is bounded by
	\begin{align*}
		(T_t K_j(x - \cdot))^2(0) & \lesssim \|(T_t K_j(x - \cdot))^2\|_{\mC^\varepsilon_\infty(\RR, e(2k+2T))} \\
		& \lesssim \|T_t K_j(x - \cdot)\|^2_{\mC^\varepsilon_\infty(\RR, e(k+T))} 
		\lesssim 2^{2j\zeta} t^{\zeta-1 - 2\varepsilon} e^{-2k|x|^\sigma}.
	\end{align*}
	Note that this estimate is much worse than the first one (because $t \in
	[0,T]$ is bounded above). We plug both those estimates
	into~\eqref{eqn:existence_of_density-pr1} and set $\zeta = -\beta -
	\varepsilon$ and \(k> {-} l\) to obtain \( \EE\big[ \|\mu(t) \|_{B^\beta_{2,2}(e(l))}^2\big] \lesssim
	t^{-\beta-1 - 3\varepsilon}\) for $\beta < 1/2$ and for $l \in
	\RR$. So finally for $p \in [1,2)$
	\[
		\EE\big[\|\mu\|_{L^p ( [0, T]; B^{\beta}_{2,2}(\RR,
		e(l)))}^p\big] = \int_0^T \EE\big[\|\mu(t)
		\|_{B^\beta_{2,2}(e(l))}^p\big] \ud t \lesssim \int_0^T
		t^{(-\beta-1 - 3\varepsilon)\frac{p}{2}} \ud t,
	\]
	and now it suffices to note that there exists $\varepsilon > 0$ with
	$(-\beta-1 - 3\varepsilon)\frac{p}{2} > -1$ if and only if $p <
	2/(\beta+1)$.
\end{proof}

\begin{corollary}
	In the setting of Proposition~\ref{lem:existence_of_density} we have
	almost surely $\sqrt{\mu} \in L^2([0,T] ; L^2(\RR,e(l)))$ for all $T >
	0$ and $l \in \RR$.
\end{corollary}

\begin{proof}[Proof of Theorem \ref{thm:rsbm_SPDE}] 
We follow the approach of Konno and Shiga \cite{KonnoShiga1988}. 
Applying Corollary~\ref{cor:convergence_total_law} for $\kappa \in (0,1/2]$ or Proposition~\ref{prop:convergence_mixiture_dawson} for $\kappa > 1/2$, we obtain an SBM in static random environment $\mu_p$, which is a process on \( (\Omega^p \times \DD([0, T]; \mM(\RR)), \mF, \PP^p \ltimes
\PP^{\omega^p}),\) with $\mF$ being the product sigma algebra. Enlarging the
probability space, we can moreover assume that the process is defined on
\((\Omega^p \times \bar{\Omega}, \mF^p \otimes \bar{\mF}, \PP^p \ltimes
\bar{\PP}^{\omega^p})\)
such that the probability space $( \bar{\Omega}, \bar{\mF}, \bar{\PP})$ supports
a space-time white noise $\bar{\xi}$ which is independent of \(\xi\). More
precisely, we are given a map 
\( \overline{\xi} : \Omega^p \times \overline{\Omega} \to \mS^\prime ( \RR^d
\times [0,T])  \)
which has the law of space-time white noise and does not depend on \(\Omega^p\),
i.e. \( \overline{\xi}(\omega^p, \overline{\omega}) = \overline{\xi}(
\overline{\omega})\).

For $\omega^p \in \Omega^p$ let $\{ \mF^{\omega^p}_t\}_{t \in [0, T]}$ be the
usual augmentation of the (random) filtration generated by $\mu(\omega^p,
\cdot)$ and $\bar{\xi}$. For almost all $\omega^p \in \Omega^p$ the collection
of martingales \(t \mapsto L^{\varphi}(\omega^p, t)\) for \(t \in [0, T], \
\varphi \in \mD_{\mH^{\omega^p}}\) defines a (random) worthy orthogonal
martingale measure $M(\omega^p, \ud t ,\ud x)$ in the sense of \cite{Walsh1986},
with quadratic variation $Q(A\times B \times [s,t]) = \int_s^t \mu(r)(A\cap B) \ud r$
for all Borel sets $A,B \subset \RR$ (first we define $Q(\varphi \times \psi
\times [s,t]) = \int_s^t \langle \mu(r), \varphi \psi \rangle \ud r$ for
$\varphi, \psi \in \mD_{\mH^{\omega^p}}$, then we use
Lemma~\ref{lem:existence_of_density} with $p = 1$ and $\beta \in (0,1/2)$ to
extend the quadratic variation and the martingales to indicator functions of
Borel sets). We can thus build a space-time white noise $\tilde{\xi}$ by
defining
for $\varphi \in L^2([0, T]\times \RR)$:
\begin{align*}
  \int_{[0,T]\times \RR} \tilde{\xi}(\omega^p, \ud s , \ud x) \varphi(s,x) : = &
  \int_{[0,T]\times \RR}\frac{M(\omega^p, \ud s , \ud x)
  \varphi(s,x)}{\sqrt{\mu(\omega^p, s,x)}} 1_{\{\mu(\omega^p, s,x)>0\}} \\
  & + \int_{[0,T]\times \RR} \bar{\xi}(\ud s , \ud x ) \varphi(s,x)
  1_{\{\mu(\omega^p, s,x) = 0\}}.  
\end{align*} 
By taking conditional expectations with respect to $\xi^p$ we see that
$\tilde{\xi}$ and $\xi^p$ are independent, and by definition the SBM in static random environment solves the SPDE~\eqref{eqn:spde_rsbm_formal_with_kappa}.

Conversely, it is straightforward to see that any solution to the SPDE is a SBM in static random environment of parameter \(\nu= \kappa/2\).  Uniqueness in law of the latter then implies
uniqueness in law of the solution to the SPDE.  \end{proof}

\subsection{Persistence}

In this section we study the persistence of the SBM in static random
environment $\mu_p$ and we prove Theorem~\ref{thm:persistence}, i.e. that
$\mu_p$ is super-exponentially persistent. For the proof we rely on the related
work \cite{Rosati2019kRSBM} which constructs, for integer $L>0$, a killed SBM in
static random environment $\mu^L_p$, in which particles are killed once they
leave the box $(-L/2,L/2)^d$. The processes $\mu^L_p$ are coupled with $\mu_p$ so that
almost surely $\mu^L_p \le \mu_p$ for all \(L\). In particular, the
following result holds.

\begin{lemma}\label{lem:killed_rsbm_construction}

  Let \(\bar{\mu}_{p}\) be an rSBM associated to a random environment
  \(\{\xi^{n}_{p}\}_{n\in \NN}\) satisfying Assumption \ref{assu:noise}. There
  exists a probability space of the form  \( (\Omega^{p} \times \DD([0, {+}
  \infty) ; \mM(\RR^{d})), \mF^{p}, \PP^{p} \ltimes \PP^{\omega^{p}})\)
  supporting a rSBM \(\mu_{p}\) such that \(\mu_{p} =
  \overline{\mu}_{p}\) in distribution. Moreover \(\Omega^{p}\) supports a
  spatial white noise \(\xi_{p}\) and there exists a null set \(N_{0} \subseteq
  \Omega^{p}\) such that:
  
  \begin{enumerate}
    \item For all \(\omega \in N_{0}^{c}\) and \(L \in 2 \NN\) the random Anderson
      Hamiltonian associated to \(\xi_{p}\) with Dirichlet boundary conditions
      on \(( {-} L /2,L/2)^{d}\), \( \mH^{\omega^{p}}_{\mf{d}, L}\), on the
      domain \(\mD_{\mH^{\omega^{p}}_{\mf{d}, L}}\) is well defined (cf.
      \cite{Chouk2019}). Moreover,
      \(\mD_{\mH^{\omega^{p}}_{\mf{d}, L}} \subseteq C^{\vt}( ({-} L/2,
      L/2)^{d})\) for any \(\vt <2 {-} d/2\). Finally the operator has discrete
      spectrum. If \(\lambda(\omega^{p},
      L) \geq 0\) is the largest eigenvalue of \(\mH^{\omega^{p}}_{\mf{d},L}\),
      then the associated eigenfunction $e_{\lambda(\omega^p, L)}$ satisfies
      $e_{\lambda(\omega^p, L)}(x) > 0$ for all $x \in (-\frac{L}{2},
      \frac{L}{2})^d$.

    \item There exist random variables \(\{\mu_{p}^{L}\}_{L \in 2 \NN}\)
      with values in \(\DD([0,\infty); \mM(\RR^{d}))\) 
      satisfying \(\mu_{p}^{L}(\omega^{p}, t) \leq \mu_{p}^{L
      {+} 2}( \omega^{p}, t) \leq \dots \leq \mu_{p}(\omega^{p}, t)\) and
      \(\mu_{p}^{L}(0) = \delta_{0}\). Moreover, for
      all \(\omega \in N_{0}^{c}\) and \(\varphi \in
      \mD_{\mH^{\omega^{p}}_{\mf{d}, L}}\):
      \[ K^{\varphi}_{L}(\omega^{p}, t) = \langle \mu_{p}^{L}(t), \varphi \rangle
      {-} \langle \mu_{p}( \omega^{p}, 0), \varphi \rangle {-} \int_{0}^{t} \ud
       r \langle \mu(r), \mH^{\omega^{p}}_{\mf{d}, L} \varphi \rangle, \ \ t \geq 0\]
      is a continuous centered martingale (w.r.t. the filtration generated by
      \(\mu^{L}_{p}(\omega^{p}, \cdot)\)) with quadratic
      variation $\langle K^{\varphi}_{L} \rangle_{t} = 2 \nu \int_{0}^{t} \ud r \langle
      \mu(r), \varphi^{2} \rangle$.

  \end{enumerate}

\end{lemma}

\begin{proof}
  For the first point see \cite{Chouk2019} and \cite[Lemma
  2.4]{Rosati2019kRSBM}.  The second statement is proved in
  \cite[Corollary 3.9]{Rosati2019kRSBM} 
\end{proof}
 
Analogously to the previous section we denote with \(t \mapsto
T^{\mf{d}}_{t}\) the semigroup associated to \(\mH^{\omega^{p}}_{\mf{d}, L}\)
for some fixed \(L, \omega^{p}\) which will be clear from the context.
Now we shall prove that given a nonzero positive $\varphi \in
C^\infty_c(\RR^d)$ and $\lambda >0$, for almost all $\omega^p$ there exists $L =
L(\omega^p)$ with
\begin{equation}\label{eqn:persistence-condition}
	\PP^{\omega^p}\big( \lim_{t \to \infty} e^{-t\lambda}\langle \mu^L_p(\omega^p,t,\cdot), \varphi \rangle = \infty \big) > 0.
\end{equation}
This implies Theorem~\ref{thm:persistence}.

The reason for working with $\mu^L_p$ is that the spectrum of the Anderson
Hamiltonian on $(-L/2,L/2)^d$ is discrete, and its
largest eigenvalue almost surely becomes bigger than $\lambda$ for $L \to
\infty$. Given this information,~\eqref{eqn:persistence-condition} follows from a simple martingale
convergence argument, see Corollary~\ref{cor:persistence_rsbm} below.

\begin{remark} 

  For simplicity we only treat the case of (killed) rSBM
  with parameter $\nu \in (0, 1/2]$. For $\nu>1/2$ we need to use the constructions of
  Section~\ref{sectn:mixing_dawson_watanabe}, after which we
  can follow the
  same arguments to show persistence.

\end{remark}

Let us write $\lambda(\omega^p, L)$ for the largest eigenvalue of the Anderson Hamiltonian $\mH_{\mathfrak d,L}^{\omega^p}$ with Dirichlet boundary
conditions on $(-L/2,L/2)^d$.

\begin{lemma}\label{lem:eigenvalue-LLN} 
  There exist $c_1, c_2>0$ such that for almost all
  $\omega^p \in \Omega^p$:
  \begin{itemize}
	  \item[(i)] In $d=1$ (by \cite[Lemmata 2.3 and 4.1]{Chen2014}):
		  \[
			  \lim_{L \to {+} \infty} \frac{\lambda(\omega^p, L)}{\log (L)^{2/3}} = c_1.
		  \]
	  \item[(ii)] In $d=2$ (by \cite[Theorem 10.1]{Chouk2019}):
		  \[
			  \lim_{L \to {+}\infty} \frac{\lambda(\omega^p, L)}{\log(L)} = c_2.
		  \]
  \end{itemize}
\end{lemma}

%

%


\begin{corollary}\label{cor:persistence_rsbm} 
Let $d \le 2$ and $\lambda >0$ and let $\mu_p$ be an SBM in static random
environment, coupled for all $L \in 2\NN$ to a killed SBM in static random
environment $\mu_p^L$ on $[-\frac{L}{2},\tfrac{L}{2}]^d$ with $\mu_p^L \le
\mu_p$ (as described in Lemma \ref{lem:killed_rsbm_construction}). For
almost all $\omega^p \in \Omega^p$a there exists an $L_0(\omega^p)>0$ such that
for all $L\geq L_0(\omega^p)$ the killed SBM $\mu_{p}^L(\omega^p, \cdot)$
satisfies~\eqref{eqn:persistence-condition}. In particular, for almost all
$\omega^p \in \Omega^p$ the process $\mu_p(\omega^p, \cdot)$ is
super-exponentially persistent.  
\end{corollary}

\begin{proof} 
  In view of Lemma~\ref{lem:eigenvalue-LLN},
  for almost all $\omega^p \in \Omega^p$ we can choose $L_0(\omega^p)$ such that
  the largest eigenvalue of the Anderson Hamiltonian $\lambda(\omega^p, L) $ is
  bigger than $\lambda$ for all $L \geq L_0(\omega^p)$. Now we fix $\omega^p$
  such that the above holds true and thus drop the index $p$ (i.e.: we will use
  a purely deterministic argument). We also fix some $L \geq L_0(\omega^p)$ and
  write $\lambda_1$ instead of $\lambda(\omega^p, L)$ for the largest
  eigenvalue. Finally, let $e_1$ be the strictly positive eigenfunction with
  $\|e_1\|_{L^2((-\frac{L}{2},\frac{L}{2})^d)}=1$ associated to $\lambda_1$. By
  Lemma~\ref{lem:killed_rsbm_construction} we find for $0 \le s < t$:
  \[
	  \EE[\langle \mu^L(t), e_1\rangle |\mF_s] = \langle \mu^L(s),
	T_{t-s}^{\mathfrak d} e_1 \rangle = \langle \mu^L(s), e^{(t-s)\lambda_1}
	e_1 \rangle,
  \]
  and thus the process $E(t) = \langle \mu^L(t), e^{-\lambda_1 t} e_{1} \rangle$,
  $t \ge 0$, is a martingale. Moreover, the variance of this martingale is
  bounded uniformly in $t$. Indeed:
    \[ \EE\big[ |E(t) {-} E(0)|^2\big] \simeq \int_0^t \ud r \ T^{\mfd}_r
    ((e^{-\lambda_1 r} e_{1})^2) (0) \lesssim \int_0^t  \ud r \ e^{-\lambda_1
  r} \lesssim 1,\]
where we used that by
    Lemma \ref{lem:killed_rsbm_construction} we have \(e_{1} \in
    \mC^{\vt}( ({-} \frac{L}{2}, \frac{L}{2} )^{d})\) for some admissible \(\vt> 0\), and therefore
    \begin{align*}   T^{\mfd}_r ((e^{-\lambda_1 r} e_{1})^2) (0) & \leq \| e_{1}
      \|_{\infty} e^{{-} \lambda_{1}r} T_{r}^{\mf{d}}( e^{{-} \lambda_{1}r} e_{1}) (0) \\
      & = \| e_{1}
  \|_{\infty} e^{{-} \lambda_{1}r} e_1(0) \lesssim e^{{-} \lambda_{1} r}. 
  \end{align*} 
It follows that $E(t)$ converges almost surely and in $L^2$ to a random variable
$E(\infty) \geq 0$ as $t\to \infty$, and since $\EE[E(\infty)] = E(0) = e_1(0) >
0$ we know that $E(\infty)$ is strictly positive with positive probability. For
$\varphi \geq 0$
nonzero with support in $[{-}L/2, L/2]^d$ we show in Lemma~\ref{lem:persistence-auxiliary} that:

\begin{equation}\label{eqn:proof_survival_phi}
  e^{-\lambda_1 t}\langle \mu^L(t), \varphi \rangle \to \langle e_{1}, \varphi
  \rangle E(\infty) , \ \ \text{ as } t \to \infty, \ \  \text{ in } \
  L^2(\PP^{\omega^p})
\end{equation}
so that we get from the strict positivity of $e_1$ and from the fact that $\lambda_1 > \lambda$
\[\PP\big( \lim_{t \to \infty} e^{-\lambda t} \langle \mu^L(t), \varphi \rangle = \infty 
\big) \geq \PP(E(\infty) > 0) >0.\]
This completes the proof.
\end{proof}

\begin{lemma}\label{lem:persistence-auxiliary}
	In the setting of Corollary~\ref{cor:persistence_rsbm}, let $\varphi \in \mC^\vt_{\mf d}$ and let \(\psi = \varphi -
	\langle e_{1}, \varphi \rangle e_{1}\). Then
\begin{equation}\label{eqn:proof_survival_second}
  \lim_{t \to \infty} \EE^{\omega^p} \Big[ |e^{- \lambda_{1} t} \langle \mu^{L}_p (\omega^p, t), \psi \rangle|^{2}
  \Big] = 0. 
\end{equation}
\end{lemma}

\begin{proof}
As before we omit the subscript $p$ from the notation, as well as the dependence on the realization $\omega^p$ of the noise. Using the martingale $\langle \mu^L(s), T^{\mf{d}}_{t-s} \psi \rangle$, we get
\begin{equation}\label{eq:persistence-auxiliary-pr}
  \EE \Big[ |\langle \mu^{L} (t), \psi \rangle|^{2} \Big]
    \lesssim | T^{\mf{d}}_{t} ( \psi) |^{2} (0) +
    \int_{0}^{t} \ud r \ T^{\mf{d}}_{r} \big[ (
    T^{\mf{d}}_{t - r}\psi)^{2} \big] (0).
\end{equation}
Let \(\lambda_{2}< \lambda_{1}\) be the second eigenvalue of the
Anderson Hamiltonian (the strict inequality is a consequence of the Krein-Rutman
theorem, cf. \cite[Lemma 2.4]{Rosati2019kRSBM}).  The main idea is to leverage that:
  \[ \| T^{\mf{d}}_{t} \psi \|_{L^{2}} \leq e^{\lambda_{2}t} \| \psi
  \|_{L^{2}},\]
since \(\psi\) is orthogonal to the first eigenfunction. The only subtlety is
that of course the value of a function in $0$ is not controlled by its $L^2$
norm. To go from $L^2$ to a space of continuous functions, we use that for all
\(\vt\) as in Equation~\eqref{eqn:condition_parameters_PAM} and sufficiently
close to \(1\):
  \begin{align*} \| T^{\mf{d}}_{1} f \|_{ \mC^{\vt}_{\mf{d}}} & \lesssim \|
  T^{\mf{d}}_{2/3} f  \|_{\mC^{\vt - \frac d 2}_{\mf{d}}} \lesssim \|
T^{\mf{d}}_{2/3} f \|_{\mC^{\vt}_{\mf{d}, 2}} \\
& \lesssim \| T^{\mf{d}}_{1/3} f \|_{\mC^{\vt- \frac d 2}_{\mf{d}, 2}} \lesssim
\| T^{\mf{d}}_{1/3} f \|_{\mC^{\vt}_{2}} \lesssim \| f \|_{L^{2}},
\end{align*}
  in view of the regularizing properties of the semigroup $T^{\mf{d}}$ (which hold with the
    same parameters as in Proposition~\ref{prop:convergence_PAM}, cf.
  \cite[Theorem 2.3]{Rosati2019kRSBM}, see also the same article for the
definition of Besov spaces with Dirichlet boundary conditions, for all current
purposes identical to the classical spaces) and by Besov embedding theorems.

Let us consider the second term in~\eqref{eq:persistence-auxiliary-pr}, for $t
\ge 2$. With the previous estimates, we bound it as follows:
\begin{align*}
\int\limits_{0}^{1} & \ud r \ T^{\mf{d}}_{r} \big[ (
    T^{\mf{d}}_{t - r} \psi)^{2} \big] (0) + \int\limits_{1}^{t} \ud r \ T^{\mf{d}}_{r} \big[ (
    T^{\mf{d}}_{t - r}\psi)^{2} \big] (0) \\
    & \lesssim \int\limits_0^1 \ud r \ \| T_{t-r}^{\mf{d}}
    \psi\|_{\mC^\vt_{\mf{d}}}^2 + \int\limits_1^{t} \ud r \  \|T^{\mf
    d}_{r-1}(T^{\mf{d}}_{t-r} \psi)^2\|_{L^2} \\
    & \lesssim \int\limits_0^1 \ud r \  \| T_{t-r-1}^{\mf{d}} \psi\|_{L^2}^2 +
    \int_{1}^{t} \ud r \ e^{\lambda_{1}(r - 1)} \| (T^{\mf{d}}_{t-r}
    \psi)^{2} \|_{L^{2}} \\
    & \lesssim \int\limits_0^1 \ud r \ \| T_{t-r-1}^{\mf{d}} \psi\|_{L^2}^2 +
    \int\limits_1^{t-1} \ud r \ e^{\lambda_1 (r-1)} \|T^{\mf{d}}_{t-r-1}
    \psi\|_{L^2}^2 + \int\limits_{t-1}^t \ud r \ e^{\lambda_1 (r-1)}
    \|\psi\|_{\mC^{\vartheta}_{\mf{d}}}^2 \\
    & \lesssim \int\limits_0^t \ud r \ e^{2\lambda_2(t-r) + \lambda_1 r} \lesssim  e^{2\lambda_2 t} (1 + e^{(\lambda_1 - 2\lambda_2)t}+t) \lesssim (e^{2\lambda_2 t} + e^{\lambda_1 t})(1+t),
\end{align*}
where we used that for any $\lambda\in\RR$ one can bound
\(\int_0^te^{\lambda s} \ud s \leq \frac{1}{|\lambda|}(1+e^{\lambda t} + t).\)
Plugging this estimate into~\eqref{eq:persistence-auxiliary-pr}, we obtain
\begin{align*}
 \EE \Big[ |e^{- \lambda_{1} t} \langle \mu^{L} (t), \psi \rangle|^{2} \Big]
    & \lesssim e^{- 2 \lambda_{1} t} e^{2\lambda_2 (t-1)} +
    e^{-2 \lambda_{1} t} \big(e^{2\lambda_2 t} + e^{\lambda_1 t} \big)(1+t)\\
    & \lesssim e^{- \lambda_1 t} + e^{-2(\lambda_1 - \lambda_2)t}(1+t).
  \end{align*}
  This proves \eqref{eqn:proof_survival_second}.
\end{proof}

\begin{remark} The connection of extinction or persistence of a branching
  particle system to the largest eigenvalue
  of the associated Hamiltonian is reminiscent of conditions appearing in the
  theory of multi-type Galton-Watson processes: See for example \cite[Section
  2.7]{Harris2002}. The martingale argument in our proof can be traced
  back at least to Everett and Ulam, as explained in \cite[Theorem 7b]{Harris1950}.
\end{remark}

\appendix\addtocontents{toc}{\protect\setcounter{tocdepth}{-1}}
\section{Construction of the Markov Process}\label{sectn:construction_markov_process}

This section is dedicated to a rigorous construction of the BRWRE. For
simplicity and without loss of generality we will work with $n = 1$. Since the
space $\NN_0^{\ZZ^d}$ is harder to deal with and we do not need it, we consider
the countable subspace $E = \big( \NN_0^{\ZZ^d} \big)_0$ of functions
$\eta\colon \ZZ^d \to \NN_0$ with $\eta(x) = 0,$ except for finitely many $x \in
\ZZ^d$. We endow $E$ with the distance \( d(\eta, \eta') = \sum_{ x \in \ZZ^d}
|\eta(x) {-} \eta'(x)|,\)
under which $E$ is a discrete and hence locally compact separable
metric space. Recall the notations from Section \ref{sectn:model}. Below we will construct ``semidirect product measures'' of the form $\PP^p \ltimes \PP^{\omega^p}$ on $\Omega^p \times \DD([0, +\infty); \RR)$, by which we mean that there exists a Markov kernel $\kappa$ such that for $A \subset \mF^p, B \subset \mB(\DD([0, +\infty); \RR))$:
\begin{equation}\label{eqn:definition_semidirect_product_measure}
\PP^p \ltimes \PP^{\omega^p} (A \times B) = \int_A \kappa(\omega^p, B) \ud \PP^p(\omega^p)
\end{equation}

\begin{lemma} Assume that for any $\omega^p \in \Omega^p$ the potential
  $\xi_p(\omega^p)$ is uniformly bounded and consider $\pi \in E$. There exists
  a unique probability measure $\PP_{\pi}$ on $\Omega = \Omega^p \times \DD([0,
  {+}\infty); E)$ endowed with the product sigma algebra, such that $\PP_{\pi}$
  is of the form $\PP^p  \ltimes \PP^{\omega^p}_{\pi}$, with
  $\PP_\pi^{\omega^p}$ being the unique measure on $\DD([0, {+}\infty); E)$
  under which the canonical process $u$ is a Markov jump process with $u(0) =
  \pi$ whose generator is given by $\mL^{\omega^p}\colon \mD(\mL^{\omega^p}) \to
  C_b(E)$, with
  \begin{equation*}\label{eqn:defn_generator_appendix} \begin{aligned}
      \mL^{\omega^{p}}&(F)(\eta) \\
      & = \sum_{x \in \ZZ^d} \eta_x \  \cdot \Big[ \Delta_{x} F(\eta) +
      (\xi_p)_{+}(\omega^p, x) d^{+}_{x} F( \eta) +
      (\xi_p)_{-}(\omega^p, x) d^{{-}}_{x}F(\eta) \Big],
  \end{aligned}
  \end{equation*} 
  where the domain $\mD(\mL^{\omega^p})$ is the set of functions $F \in C_b(E)$
  such that the right-hand side  lies in $C_b(E)$.  
\end{lemma}

\begin{proof} 
  The construction for fixed $\omega^p \in
\Omega^p$ is classical. Indeed, the generator
has the form of \cite[(4.2.1)]{EthierKurtz1986}, with \(\lambda(\eta) = \sum_{x \in \ZZ^d }\eta_x(2d
{+} |\xi_p|(\omega^p, x)),\) and we only need to rule out explosions by
verifying that almost surely \(\sum_{k \in \NN} \frac{1}{\lambda(Y_k)} ={+}
\infty\), where \(Y\) is the associated discrete time Markov chain.
This is the case, since $\xi_p$ is bounded and thus
\[ \sum_{k \in \NN} \frac{1}{\lambda(Y_k)} \gtrsim \sum_{k \in \NN}
\frac{1}{\sum_x Y_k(x)} \ge \sum_{k \in \NN}\frac{1}{c {+} k} = {+} \infty\]
with $c = \sum_x \pi(x).$ It follows via classical calculations that
$\mL^{\omega^p}$ is the generator associated to the process $u$.  This allows us
to define
for fixed $\omega^p$ the law $\kappa(\omega^p, \cdot)$ of our process on $\DD([0,
{+} \infty); E)$. To construct the measure $\PP_{\pi}$ we have to show that $\kappa$ is a
Markov kernel, which amounts to proving measurability in the $\omega^p$
coordinate. But $\kappa$ depends continuously on $\xi_p$, which we can
verify by coupling the processes for $\xi_p$ and $\tilde \xi_p$
through a construction based on Poisson jumps at rate $K > \| \xi_p\|_\infty, \|
\tilde \xi_p\|_\infty$ and then rejecting the jumps if an independent uniform
$[0,K]$ variable is not in $[0,|\xi_p(x)|]$ respectively in $[0,|\tilde
\xi_p(x)|]$. Since $\xi_p$ is measurable in $\omega^p$, also $\kappa$ is
measurable in $\omega^p$.
\end{proof}

Next, we extend the construction to potentials of sub-polynomial growth: 

\begin{lemma}\label{lem:appendix_construction_markov_process} Let $\xi_p(\omega^p)\in \bigcap_{a>0} L^{\infty}(\ZZ^d, p(a))$
  for all $\omega^p \in \Omega^p$ and consider $\pi \in E$. There exists
  a unique probability measure $\PP_{\pi} = \PP^p  \ltimes \PP^{\omega^p}_{\pi}$ on $\Omega = \Omega^p \times \DD([0,
  {+}\infty); E)$ endowed with the product sigma algebra, where
  $\PP_\pi^{\omega^p}$ is the unique measure on $\DD([0, {+}\infty); E)$
  under which the canonical process $u$ is a Markov jump process with $u(0) =
  \pi$ and with generator $\mL^{\omega^p}$ and $\mD(\mL^{\omega^p})$ defined as in the
  previous lemma.  
\end{lemma}

\begin{proof} 

Let us fix $\omega^p \in \Omega^p$. Consider the Markov jump processes $u^k$
started in $\pi$ with generator $\mL^{\omega^p, k}$ associated to $\xi_p^k(x) = (\xi_p(x)\wedge k)\vee ({-}k)$ whose existence
follows from the previous result. The sequence $\{ u^k \}_{k \in \NN}$  is tight
(this follows as in Lemma~\ref{lem:one_dimensional_tightness} and
Corollary~\ref{cor:tightness_of_mu_n}, keeping $n$ fixed but letting $k$ vary)
and converges weakly to a Markov process $u$. 
Indeed, for $k,R \in \NN$ let $\tau^k_R$ be the first time with
$\supp(u^k(\tau^k_R)) \not\subset Q(R)$, where $Q(R)$ is the square of radius
$R$ around the origin, and let $\tau_R$ be the corresponding exit time for $u$.
Then we get for all $k > \max_{x \in Q(R)} |\xi_p(x)|$, for all $T>0$, and all
$F \in C_b(\DD([0,T];E))$:
\[
	\EE^{\omega_p}_{\pi}[F((u^k(t))_{t \in [0,T]}) 1_{\{\tau^k_R \le T\}}]= \EE^{\omega_p}_{\pi}[F((u(t))_{t
	\in [0,T]}) 1_{\{\tau_R \le T\}}],
\]
where we used that the exit time $\tau_R$ is continuous because $E$ is a
discrete space. Moreover, from the tightness of $\{ u^k \}_{k \in \NN}$ it
follows that for all $\varepsilon > 0$ and $T>0$ there exists $R \in \NN$ with
$\sup_k \PP(\tau^k_R \le T) < \varepsilon$. This proves the uniqueness in law and that
$u$ is the limit (rather than subsequential limit) of $\{ u^k \}_{k \in \NN}$.
Similarly we get the Markov property of $u$ from the Markov property of the $\{
u^k \}_{k \in \NN}$ and from the convergence of the transition functions.

 It remains to verify that $\mL^{\omega^p}$ is the generator of $u$. But for
 large enough $R$ we have $\PP^{\omega^p}_{\pi} (\tau_R \le h) = O(h^2)$ as $h
 \to 0^+$, because on the event $\{\tau_R \le h\}$ at least two transitions must
 have happened (recall that $\pi$ is compactly supported). 
We can thus compute for any $F \in C_b(E)$: $$ \EE^{\omega^p}_{\pi} \big[
F(u(h))\big] = \EE^{\omega^p}_{\pi} \big[ F(u^k(h))\big] + O(h^2).  $$ The
result on the generator then follows from the previous lemma. As before, we now
have constructed a collection of
probability measures $\kappa(\omega^p, \cdot)$ as the limit of the Markov
kernels $\kappa^k(\omega^p, \cdot).$ Since measurability is preserved when
passing to the limit, this concludes the proof.
\end{proof} 

\section{Some Estimates for the Random Noise}

In this section we prove parts of Lemma~\ref{lem:renormalisation}, i.e. that a
random environment satisfying Assumption~\ref{assu:noise} gives rise to a
deterministic environment satisfying Assumption~\ref{assu:renormalisation}.

\begin{lemma}\label{lem:mean_noise} 
Let $a, \varepsilon,q > 0$ and $b> d/2$.  Under Assumption \ref{assu:noise} we
have
\[
	\sup_n \bigg[ \EE
	\|n^{-d/2}(\xi^n_p)_{+}\|_{\mC^{{-}\varepsilon}(\ZZ^d_n, p(a))}^q  + \EE
      \| n^{-d/2}(\xi^n_p)_+\|_{L^2(\znd, p(b))}^2 \bigg] < {+} \infty,
\]
and the same holds if we replace $(\xi^n_p)_{+}$ with $|\xi^n_p|$.
Furthermore, for $\nu = \EE[\Phi_+],$ the following convergences hold true in
distribution in $\mC^{{-}\varepsilon}(\RR^d, p(a))$: 
  \begin{align*} \mathcal{E}^n n^{-d/2}(\xi^n_p)_+ \longrightarrow \nu, \qquad
  \mathcal{E}^n n^{-d/2}|\xi^n_p|  \longrightarrow 2\nu.  \end{align*}
 \end{lemma}

\begin{proof}

 We prove the result only for $(\xi^n_p)_{+}$, since then we can treat
 $(\xi^n_p)_-$ by considering $-\xi^n_p$ ($-\Phi$ is still a centered
 random variable). Now note that we can rewrite \(\EE[\| n^{-d/2}(\xi^n_p)_+
 \|_{L^q(\znd,p(a))}^q]\) as 
 \begin{align*}
	\sum_{x \in \znd} n^{-d} \EE[|n^{-d/2}(\xi^n_p)_+|^q (x) ] |p(a)(x)|^q
	\lesssim \EE[|\Phi|^q] \int_{\RR^d} (1+|y|)^{-aq} \ud y,
 \end{align*}
 which is finite whenever $aq>d$. From here the uniform bound on the
 expectations follows by Besov embedding.

 Convergence to \(\nu\) is then a consequence of the spatial independence
 of the noise \(\xi^{n}\), since it is easy to see that \(\EE \big[
 \langle \mE^{n} (\xi^{n}_{p})_{+} {-} \nu, \varphi \rangle \big] =
 O(n^{{-}d})\) for all \(\varphi\) with compactly supported Fourier transform.
\end{proof}

The following result is a simpler variant
of~\cite[Lemma~5.5]{MartinPerkowski2017} for the case $d=1$, hence we omit the
proof.
\begin{lemma}\label{lem:convergence_1d_white_noise} Fix $\xi^n$ satisfying
Assumption \ref{assu:noise}, $d = 1$, $a,q >0$ and $\alpha< 2{-}d/2$. We have:
  \[
	  \sup_n \EE\big[ \|\xi^n_p \|_{\mC^{\alpha{-}2}(\znd, p(a))}^q \big] <
	{+}\infty, \qquad \mE^n \xi^n_p \to \xi_p,
  \]
where $\xi_p$ is a white noise on $\RR$ and the convergence holds in
distribution in $\mC^{\alpha{-}2}(\RR^d, p(a)).$ \end{lemma}

\section{Moment Estimates}

Here we derive uniform bounds for the moments of the processes $\{\mu^n\}_{n\in
\NN}$. As a convention, in the following we will write $\EE$ and $\PP$ for the
expectation and the probability under the distribution of $u^n$. For different
initial conditions $\eta \in E$  we will write $\EE_{\eta}, \PP_{\eta}$. 

\begin{lemma}\label{lem:moments_estimate}

  Fix \(q, T>0\). For all $n \in \NN$, consider the process $\{\mu^n(t)\}_{t \ge 0}$  as in
  Definition~\ref{def:of_u_n_and_mu_n}. Consider then $\varphi^n\colon \znd
  \to \RR$ with $\varphi^n \ge 0$, $\varphi^n = \varphi \vert_{\znd}$ with $\varphi \in
  \mC^2(\RR^d, e(l))$ for some $l \in \RR$. Then
  \[
	  \sup_{n} \sup_{t \in [0, T]} \EE \big[ |\mu^n(t)(\varphi^n)|^q \big] <
	{+} \infty.
  \]
If for all $\ve>0$ there exists an $l \in \RR$
  such that $ \sup_n \| \varphi^n \|_{\mC^{-\ve}(\RR^d, e(l))} <{+} \infty $, we
  can bound for all $\gamma \in (0,1)$: 
  \[ \sup_{n} \sup_{t \in
  [0, T]} t^{\gamma} \EE \big[ |\mu^n(t)(\varphi^n)|^q \big] < {+} \infty.  \]
\end{lemma}

\begin{proof} 
  We prove the second estimate, since the first estimate is similar but easier
  (Lemma~\ref{lem:restriction_to_lattice} below controls $\| \varphi^n
  \|_{\mC^\vartheta(\znd, e(l))}$ for all $\vartheta < 2$ in that case). Also,
  we assume without loss of generality that $q \ge 2$. As usual, we use the
  convention of freely increasing the value of $l$ in the exponential weight.
  Let us start by recalling that $\EE\big[ \mu^n(t)(\varphi^n) \big]
  =T^n_{t}\varphi^n (0)$. Moreover, via the assumption on the regularity,
    Proposition~\ref{prop:convergence_PAM} and Equation
    \eqref{eqn:loose_time_space_regularity} from
    Lemma~\ref{lem:weighted_paraproduct_estimates} guarantees that for any
    $\gamma \in (0,1)$ there exists a $ \delta= \delta(\gamma,q)>0$ such that
  \[\sup_n \| t
  \mapsto T^n_t \varphi^n\|_{\mL^{\gamma/q, \delta}(\znd, e(l))} < {+}
\infty.\]
By the triangle inequality it thus suffices to prove that for any $\gamma>0$: 
\[ \sup_{n} \sup_{t \in [0, T]} t^{\gamma} \EE \big[
  |\mu^n(t)(\varphi^n){-}  T^n_{t}\varphi^n (0)|^q \big] < {+} \infty.  \]
  
  Note that we can interpret the particle system $u^n$ as the superposition of
  $\fn$ independent particle systems, each started with one particle in zero; we
  write $u^n = u^n_1 + \dots + u^n_{\fn}$. To lighten the notation we assume
  that $n^{\vr} \in\NN$. We then apply Rosenthal's inequality,
  \cite[Theorem~2.9]{Petrov1995} (recall that $q\ge 2$) and obtain (with
  \( (f, g) = \sum_{x \in \znd} f(x)g(x)\)):
  \begin{align*} 
    &\EE \big[ |\mu^n(t)(\varphi^n){-} T^n_{t}\varphi^n (0)|^q \big]  = \EE
    \bigg[ \bigg| \sum_{k=1}^{n^\vr} \big[ n^{-\vr}(  u^n_k(t), \varphi^n ) {-}
    n^{-\vr}T^n_{t}\varphi^n (0) \big] \bigg|^q \bigg] \\ 
    &\lesssim n^{-\vr q} \sum_{k=1}^{n^\vr}  \EE\big[ |(  u^n_k(t), \varphi^n)
    {-} T^n_{t}\varphi^n (0)|^q\big] \\
    & \qquad + n^{-\vr q} \bigg( \sum_{k=1}^{n^\vr}
    \EE\big[ |(  u^n_k(t), \varphi^n ) {-} T^n_{t}\varphi^n (0)|^2\big]
    \bigg)^{\frac{q}{2}} \\ 
    & \lesssim n^{-\vr (q-1)} \EE\big[ |(  u^n_1(t), \varphi^n )|^q \big] +
    \big(n^{-\vr} \EE\big[ |(  u^n_1(t), \varphi^n ) |^2\big]\big)^{q/2} \\
   & \qquad + n^{-\frac{\vr q}{2}} t^{-\gamma} \| t \mapsto T^n_t \varphi^n\|_{\mL^{\gamma/q,
    \delta}(\znd, e(l))}^q
  \end{align*}
  for the same $\delta > 0$ and $l \in \RR$ as above.  The two scaled
  expectations are of the same form, in the second term we simply have $q=2$. To
  control them, we define for $p \in \NN$ the map 
    \[
    m^{p,n}_{\varphi^n}(t,x) = n^{\vr(1- p)}\EE_{1_{\{x\}}} \big[ | (u^n_1(t),
  \varphi^n)|^p \big].  \]
  As a consequence of Kolmogorov's backward equation each $m^{p,n}_{\varphi^n}$
  solves the discrete PDE (see also Equation~(2.4) in
  \cite{AlbeverioBogachevEtAl2000}):
   \begin{align*}
	\partial_t m^{p, n}_{\varphi^n}(t,x) = \mH^n m^{p, n}_{\varphi^n}(t,x) +
	n^{-\vr} (\xi^n_e)_+(x) \sum_{i = 1}^{p-1} \binom{p}{i} m^{i,
	n}_{\varphi^n}(t,x) m^{p{-}i,  n}_{\varphi^n}(t,x),
  \end{align*}
  with initial condition $m^{p, n}_{\varphi^n}(0,x) =
  n^{\vr(1-p)}|\varphi^n(x)|^p$. We claim that this equation has a unique
  (paracontrolled in $d=2$) solution $m^{p, n}_{\varphi^n}$, such that for all
  $\gamma > 0$ there exists $\delta = \delta(\gamma,p) > 0$ with
  $\sup_n \|m^{n,p}_{\varphi^n} \|_{\mL^{\gamma, \delta}(\znd, e(l))} < \infty$.
  Once this is shown, the proof is complete. We proceed by induction over $p$.
  For $p=1$ we simply have $m^{n,1}_{\varphi^n}(t,x) = T^n_t \varphi^n(x)$. For
  $p \ge 2$ we use that by Lemma~\ref{lem:conv_to_zero} we have
  $\|n^{\vr(1-p)}|\varphi^n(x)|^p\|_{\mC^{\kappa}(\znd, e(l))} \to 0$ for some
  $\kappa>0$ and we assume that the induction hypothesis holds for all $p'<p$.
  Since it suffices to prove the bound for small $\gamma>0$, we may assume also
  that $\kappa>\gamma$. We choose then $\gamma' < \gamma$ such that for some
  $\delta(\gamma', p)>0$:
  \[
    \sup_n \bigg\| \sum_{i = 1}^{p-1} m^{i, n}_{\varphi^n} m^{p{-}i,
    n}_{\varphi^n}\bigg\|_{\mM^{\gamma'}\mC^{\delta(\gamma', p)}(\znd, e(l))} < {+}
    \infty.  
  \]
  Since by Assumption \ref{assu:renormalisation}
  $\|n^{-\vr}(\xi^n_e)_+\|_{\mC^{-\varepsilon}(\znd,p(a))}$ is uniformly bounded
  in $n$ for all $\ve, a >0,$ the above bound is sufficient to control the
  product: 
  \[ 
    \sup_n \bigg\| n^{-\vr}(\xi^n_e)_+\sum_{i = 1}^{p-1} m^{i, n}_{\varphi^n}
    m^{p{-}i,n}_{\varphi^n}\bigg\|_{\mM^{\gamma'} \mC^{-\ve}(\znd, e(l))} < {+}
    \infty.
  \]
  Now the claimed bound for $m^{n,p}_{\varphi^n}$ follows from an application of
  Proposition~\ref{prop:convergence_PAM}. For non-integer $q$ we simply use
  interpolation between the bounds for $p < q < p'$ with $p,p' \in \NN$.
\end{proof}

\section{Some Estimates in Besov Spaces}

Here we prove some results concerning discrete and continuous Besov spaces.
First, we show that restricting a function to the lattice preserves its
regularity.

\begin{lemma}\label{lem:restriction_to_lattice} 
  Let $\varphi \in \mC^{\alpha}(\RR^d )$ for $\alpha \in \RR_{> 0}\setminus
  \NN$. Then $\varphi\vert_{\znd} \in \mC^{\alpha}(\znd)$ and $$ \sup_{n \in
  \NN} \| \varphi\vert_{\znd} \|_{\mC^{\alpha}(\znd)} \lesssim \| \varphi
  \|_{\mC^{\alpha}(\RR^d)}. $$ For the extension of $\varphi\vert_{\znd}$ we
  have $\mE^n(\varphi\vert_{\znd}) \to \varphi$ in $\mC^{\beta}(\RR^d)$ for all
  $\beta < \alpha.$
\end{lemma}

\begin{proof} 
  
  Let us call $\varphi^n = \varphi \vert_{\znd}.$ We have to estimate 
  $\| \Delta^n_{j} \varphi^n \|_{L^{\infty}(\znd)}$, and for that purpose we
  consider the cases $j < j_n$ and $j = j_n$ separately. In the first case we
  have \(\Delta^n_j \varphi^n (x)= K_j * \varphi (x) = \Delta_j
  \varphi(x)\) for $x \in \znd$ because, as $\supp(\varrho_j) \subset n({-}1/2, 1/2)^d$, the
  discrete and the continuous convolutions coincide. Therefore: 
  \[ 
    \| \Delta_{j}^{n} \varphi \|_{L^{\infty}( \znd)} \leq \| \Delta_{j} \varphi
    \|_{ L^{\infty}(\RR^{d})} \leq 2^{j \alpha} \| \varphi \|_{\mC^{\alpha}}. 
  \]

  For $j = j_n$ we have $\varrho^n_{j_n}(\cdot) = 1{-}\chi(2^{{-}j_n}\cdot )$,
where $\chi \in \mS_{\omega}$ is one of the two functions generating the dyadic
partition of unity, a symmetric smooth function such that $\chi = 1$ in a ball
around the origin. By construction we have $\varrho^n_{j_n}(x) \equiv 1$ for $x$
near the boundary of $n(-1/2,1/2)^d$, and therefore $\supp(\chi(2^{{-}j_n}\cdot
)) \subset n(-1/2,1/2)^d$. Let us define $\psi_n = \mF_n^{-1}
\chi(2^{{-}j_n}\cdot ) = \mF_{\RR^d}^{-1} \chi(2^{{-}j_n}\cdot )$. Then
\[ 
  \sum_{x \in \znd} n^{-d} \psi_n(x) = \mF_n \psi_n(0) = \chi(2^{-j_n} \cdot 0)
  = 1,
\]
and for every monomial $M$ of strictly positive degree we have, since $\psi_n$
is an even function,
\[ 
  \sum_{x \in \znd} n^{-d} \psi_n(x) M(x) =(\psi_n\ast M)(0) = \mF_{\RR^d}^{-1}(
  \chi(2^{-j_n}\cdot) \mF_{\RR^d} M)(0) = M(0) = 0,
\]
where we used that the Fourier transform of a polynomial is supported in $0$. 
Thus for $x \in \znd$ we get \(\Delta^n_{j_n} \varphi^n (x) = \varphi(x) {-}
(\psi_n \ast_n \varphi)(x)\), that is:
\begin{align*}
  \varphi(x) {-} (\psi_n \ast_n \varphi)(x) = {-}
  \psi_n \ast_n \bigg( \varphi(\cdot)-\varphi(x)-\!\!\! \sum_{1\le |k| \le \lfloor
  \alpha \rfloor}\frac{1}{k!} \partial^k \varphi(x)  (\cdot{-}x)^{k} \bigg)(x),
\end{align*}
with the usual multi-index notation and where as above we could replace the discrete
convolution $\ast_n$ with a convolution on $\RR^d$. Moreover, since $\varphi \in
\mC^{\alpha}(\RR^d)$ and $\alpha > 0$ is not an integer, we can  estimate
\[
	\bigg\|\varphi(\cdot)- \!\!\! \sum_{0 \le |k| \le \lfloor \alpha
	\rfloor}\frac{1}{k!} \partial^k \varphi(x) (\cdot{-}x)^{\otimes
	k}\bigg\|_{L^{\infty}(\RR^d)} \lesssim
	|y|^{\alpha}\|\varphi\|_{\mC^{\alpha}(\RR^d)},
\]
and from here the estimate for the convolution holds by a scaling argument. The
convergence then follows by interpolation.
\end{proof}

The following result shows that multiplying a function on $\znd$ by
$n^{-\kappa}$ for some $\kappa>0$ gains regularity and gives convergence to zero
under a uniform bound for the norm.

\begin{lemma}\label{lem:conv_to_zero} Consider $z \in \weights$ and $p \in [1,
  \infty], \alpha \in \RR$ and a sequence of functions $f^n \in
  \mC^{\alpha}_p(\ZZ^d_n, z)$ with uniformly bounded norm: $$ \sup_n \| f^n
  \|_{\mC^{\alpha}_p(\ZZ^d_n, z)} < + \infty.  $$ Then for any $\kappa > 0$ the
  sequence $n^{-\kappa}f^n$ is bounded in $\mC^{\alpha{+}\kappa}_p(\znd, z)$: 
  \[
  \sup_n \| n^{-\kappa} f^n \|_{\mC^{\alpha+\kappa}_p(\znd, z)} \lesssim \sup_n
\| f^n\|_{\mC^{\alpha}_p(\znd, z)} \]
and $n^{-\kappa} \mE^n f^n$ converges to
zero in $\mC^{\beta}_p (\RD, z)$ for any $\beta < \alpha + \kappa.$
\end{lemma}

\begin{proof} 
  By definition, we only encounter Littlewood-Paley blocks up to an order
  $j_n \simeq \log_2(n).$ Hence $2^{j(\alpha {+} \kappa {-} \ve)}n^{-\kappa} \lesssim 2^{j\alpha}n^{- \ve}$ for $j \le j_n$ and $\ve \ge 0$, from where the claim follows.
\end{proof}

Now we study the action of discrete gradients. We write $\mC^{\alpha}_p(\znd, z ;
\RR^d)$ for the space of maps $\varphi\colon \znd \to \RR^d$ such that each
component lies in $\mC^{\alpha}_p(\znd,z)$ with the naturally induced norm. 

\begin{lemma}[\cite{MartinPerkowski2017}, Lemma~3.4]\label{lem:discrete_derivative}
  The discrete gradient $(\nabla^n \varphi)_i(x) = n(\varphi(x{+}\frac{e_i}{n}){-}\varphi(x))$ for
  $i = 1, \ldots, d$ (with $\{ e_i \}_{i}$ the standard basis in $\RR^d$) and
  the discrete Laplacian $\Delta^n \varphi (x) = n^2 \sum_{i = 1}^d
  (\varphi(x{+}\frac{e_i}{n}){-}2\varphi(x) {+} \varphi(x{-}\frac{e_i}{n}))$
  satisfy:
  \begin{align*}
    \| \nabla^n \varphi \|_{\mC^{\alpha{-}1}_p(\znd, z; \RR^d)} \lesssim \| \varphi
    \|_{\mC^{\alpha}_p(\znd, z)}, \qquad \| \Delta^n
    \varphi\|_{\mC^{\alpha{-}2}_p(\znd, z)} \lesssim \| \varphi\|_{\mC^{\alpha}_p(\znd, z)},
  \end{align*}
  for all $\alpha \in \RR$ and $p \in [1,\infty]$, where both estimates hold uniformly in $n \in \NN$.
\end{lemma}

\begin{proof}
	For $\Delta^n$ this is shown in \cite[Lemma~3.4]{MartinPerkowski2017}. The
	argument for the gradient $\nabla^n$ is essentially the same but
	slightly easier.  
\end{proof}

\bibliographystyle{alphaurl} 


\begin{thebibliography}{GMPV10}

\bibitem[ABMY00]{AlbeverioBogachevEtAl2000}
S.~Albeverio, L.~V. Bogachev, S.~A. Molchanov, and E.~B. Yarovaya.
\newblock Annealed moment {L}yapunov exponents for a branching random walk in a
  homogeneous random branching environment.
\newblock {\em Markov Process. Related Fields}, 6(4):473--516, 2000.

\bibitem[AC15]{Allez2015}
R.~Allez and K.~Chouk.
\newblock The continuous {A}nderson hamiltonian in dimension two.
\newblock {\em arXiv preprint arXiv:1511.02718}, 2015.

\bibitem[BGK09]{BartschGantertKochler2009}
C.~Bartsch, N.~Gantert, and M.~Kochler.
\newblock Survival and growth of a branching random walk in random environment.
\newblock {\em Markov Process. Related Fields}, 15(4):525--548, 2009.

\bibitem[Che56]{Chentsov1956Weak}
N.~Chentsov.
\newblock Weak convergence of stochastic processes whose trajectories have no
  discontinuities of the second kind and the “heuristic” approach to the
  kolmogorov-smirnov tests.
\newblock {\em Theory of Probability \& Its Applications}, 1(1):140--144, 1956.

\bibitem[Che14]{Chen2014}
X.~Chen.
\newblock Quenched asymptotics for {B}rownian motion in generalized {G}aussian
  potential.
\newblock {\em Ann. Probab.}, 42(2):576--622, 2014.
\newblock \href {https://doi.org/10.1214/12-AOP830}
  {\path{doi:10.1214/12-AOP830}}.

\bibitem[Cri04]{Crisan2004Superbrown}
D.~Crisan.
\newblock Superprocesses in a {B}rownian environment.
\newblock {\em Proc. R. Soc. Lond. Ser. A Math. Phys. Eng. Sci.},
  460(2041):243--270, 2004.
\newblock Stochastic analysis with applications to mathematical finance.
\newblock \href {https://doi.org/10.1098/rspa.2003.1242}
  {\path{doi:10.1098/rspa.2003.1242}}.

\bibitem[CT18]{Corwin2018}
I.~Corwin and L.C. Tsai.
\newblock {SPDE} limit of weakly inhomogeneous {ASEP}.
\newblock {\em arXiv preprint arXiv:1806.09682}, 2018.

\bibitem[CT19]{Chakraborty2018}
P.~Chakraborty and S.~Tindel.
\newblock Rough differential equations with power type nonlinearities.
\newblock {\em Stochastic Process. Appl.}, 129(5):1533--1555, 2019.
\newblock \href {https://doi.org/10.1016/j.spa.2018.05.010}
  {\path{doi:10.1016/j.spa.2018.05.010}}.

\bibitem[CvZ19]{Chouk2019}
K.~Chouk and W.~van Zuijlen.
\newblock Asymptotics of the eigenvalues of the {A}nderson {H}amiltonian with
  white noise potential in two dimensions.
\newblock {\em arXiv preprint arXiv:1907.01352}, 2019.

\bibitem[DMS93]{DawsonMaisonneuve1993SaintFlour}
D.~A. Dawson, B.~Maisonneuve, and J.~Spencer.
\newblock {\em \'{E}cole d'\'{E}t\'{e} de {P}robabilit\'{e}s de {S}aint-{F}lour
  {XXI}---1991}, volume 1541 of {\em Lecture Notes in Mathematics}.
\newblock Springer-Verlag, Berlin, 1993.
\newblock Papers from the school held in Saint-Flour, August 18--September 4,
  1991, Edited by P. L. Hennequin.
\newblock \href {https://doi.org/10.1007/BFb0084189}
  {\path{doi:10.1007/BFb0084189}}.

\bibitem[EK86]{EthierKurtz1986}
S.~N. Ethier and T.~G. Kurtz.
\newblock {\em Markov processes}.
\newblock Wiley Series in Probability and Mathematical Statistics: Probability
  and Mathematical Statistics. John Wiley \& Sons, Inc., New York, 1986.
\newblock Characterization and convergence.
\newblock \href {https://doi.org/10.1002/9780470316658}
  {\path{doi:10.1002/9780470316658}}.

\bibitem[Eth00]{Etheridge2000}
A.~M. Etheridge.
\newblock {\em An introduction to superprocesses}, volume~20 of {\em University
  Lecture Series}.
\newblock American Mathematical Society, Providence, RI, 2000.
\newblock \href {https://doi.org/10.1090/ulect/020}
  {\path{doi:10.1090/ulect/020}}.

\bibitem[FN77]{Fukushima1976}
M.~Fukushima and S.~Nakao.
\newblock On spectra of the {S}chr\"{o}dinger operator with a white {G}aussian
  noise potential.
\newblock {\em Z. Wahrscheinlichkeitstheorie und Verw. Gebiete},
  37(3):267--274, 1976/77.
\newblock \href {https://doi.org/10.1007/BF00537493}
  {\path{doi:10.1007/BF00537493}}.

\bibitem[GIP15]{GubinelliImkellerPerkowski2015}
M.~Gubinelli, P.~Imkeller, and N.~Perkowski.
\newblock Paracontrolled distributions and singular {PDE}s.
\newblock {\em Forum Math. Pi}, 3:e6, 75, 2015.
\newblock \href {https://doi.org/10.1017/fmp.2015.2}
  {\path{doi:10.1017/fmp.2015.2}}.

\bibitem[GKS13]{GunKonigSekulovic2013}
O.~G\"{u}n, W.~K\"{o}nig, and O.~Sekulovi\'{c}.
\newblock Moment asymptotics for branching random walks in random environment.
\newblock {\em Electron. J. Probab.}, 18:no. 63, 18, 2013.
\newblock \href {https://doi.org/10.1214/ejp.v18-2212}
  {\path{doi:10.1214/ejp.v18-2212}}.

\bibitem[GM90]{GaertnerMolchanov1990}
J.~G\"{a}rtner and S.~A. Molchanov.
\newblock Parabolic problems for the {A}nderson model. {I}. {I}ntermittency and
  related topics.
\newblock {\em Comm. Math. Phys.}, 132(3):613--655, 1990.
\newblock URL: \url{http://projecteuclid.org/euclid.cmp/1104201232}.

\bibitem[GMPV10]{GantertMullerPopov2010}
N.~Gantert, S.~M\"{u}ller, S.~Popov, and M.~Vachkovskaia.
\newblock Survival of branching random walks in random environment.
\newblock {\em J. Theoret. Probab.}, 23(4):1002--1014, 2010.
\newblock \href {https://doi.org/10.1007/s10959-009-0227-5}
  {\path{doi:10.1007/s10959-009-0227-5}}.

\bibitem[GP17]{Gubinelli2017KPZ}
M.~Gubinelli and N.~Perkowski.
\newblock K{PZ} reloaded.
\newblock {\em Comm. Math. Phys.}, 349(1):165--269, 2017.
\newblock \href {https://doi.org/10.1007/s00220-016-2788-3}
  {\path{doi:10.1007/s00220-016-2788-3}}.

\bibitem[GUZ20]{Gubinelli2018Semilinear}
M.~Gubinelli, B.~Ugurcan, and I.~Zachhuber.
\newblock Semilinear evolution equations for the {A}nderson {H}amiltonian in
  two and three dimensions.
\newblock {\em Stoch. Partial Differ. Equ. Anal. Comput.}, 8(1):82--149, 2020.
\newblock \href {https://doi.org/10.1007/s40072-019-00143-9}
  {\path{doi:10.1007/s40072-019-00143-9}}.

\bibitem[Hai14]{Hairer2014}
M.~Hairer.
\newblock A theory of regularity structures.
\newblock {\em Invent. Math.}, 198(2):269--504, 2014.
\newblock \href {https://doi.org/10.1007/s00222-014-0505-4}
  {\path{doi:10.1007/s00222-014-0505-4}}.

\bibitem[Har51]{Harris1950}
T.~E. Harris.
\newblock Some mathematical models for branching processes.
\newblock In {\em Proceedings of the {S}econd {B}erkeley {S}ymposium on
  {M}athematical {S}tatistics and {P}robability, 1950}, pages 305--328.
  University of California Press, Berkeley and Los Angeles, 1951.

\bibitem[Har02]{Harris2002}
T.~E. Harris.
\newblock {\em The theory of branching processes}.
\newblock Dover Phoenix Editions. Dover Publications, Inc., Mineola, NY, 2002.
\newblock Corrected reprint of the 1963 original [Springer, Berlin; MR0163361
  (29 \#664)].

\bibitem[HL15]{Hairer2015Simple}
M.~Hairer and C.~Labb\'{e}.
\newblock A simple construction of the continuum parabolic {A}nderson model on
  {${\bf R}^2$}.
\newblock {\em Electron. Commun. Probab.}, 20:no. 43, 11, 2015.
\newblock \href {https://doi.org/10.1214/ECP.v20-4038}
  {\path{doi:10.1214/ECP.v20-4038}}.

\bibitem[HL18]{Hairer2018Multiplicative}
M.~Hairer and C.~Labb\'{e}.
\newblock Multiplicative stochastic heat equations on the whole space.
\newblock {\em J. Eur. Math. Soc. (JEMS)}, 20(4):1005--1054, 2018.
\newblock \href {https://doi.org/10.4171/JEMS/781}
  {\path{doi:10.4171/JEMS/781}}.

\bibitem[K{\"o}n16]{Konig2016}
W.~K{\"o}nig.
\newblock {\em The parabolic {A}nderson model}.
\newblock Pathways in Mathematics. Birkh\"{a}user/Springer, [Cham], 2016.
\newblock Random walk in random potential.
\newblock \href {https://doi.org/10.1007/978-3-319-33596-4}
  {\path{doi:10.1007/978-3-319-33596-4}}.

\bibitem[KS88]{KonnoShiga1988}
N.~Konno and T.~Shiga.
\newblock Stochastic partial differential equations for some measure-valued
  diffusions.
\newblock {\em Probability Theory and Related Fields}, 79(2):201--225, Sep
  1988.
\newblock \href {https://doi.org/10.1007/BF00320919}
  {\path{doi:10.1007/BF00320919}}.

\bibitem[Lab19]{Labbe2018}
C.~Labb\'{e}.
\newblock The continuous {A}nderson {H}amiltonian in {$d\leq 3$}.
\newblock {\em J. Funct. Anal.}, 277(9):3187--3235, 2019.
\newblock \href {https://doi.org/10.1016/j.jfa.2019.05.027}
  {\path{doi:10.1016/j.jfa.2019.05.027}}.

\bibitem[MP19]{MartinPerkowski2017}
J.~Martin and N.~Perkowski.
\newblock Paracontrolled distributions on {B}ravais lattices and weak
  universality of the 2d parabolic {A}nderson model.
\newblock {\em Ann. Inst. Henri Poincar\'{e} Probab. Stat.}, 55(4):2058--2110,
  2019.
\newblock \href {https://doi.org/10.1214/18-AIHP942}
  {\path{doi:10.1214/18-AIHP942}}.

\bibitem[MX07]{MytnikXiong2007LocalExtinction}
L.~Mytnik and J.~Xiong.
\newblock Local extinction for superprocesses in random environments.
\newblock {\em Electron. J. Probab.}, 12:no. 50, 1349--1378, 2007.
\newblock \href {https://doi.org/10.1214/EJP.v12-457}
  {\path{doi:10.1214/EJP.v12-457}}.

\bibitem[Myt96]{Mytnik1996}
L.~Mytnik.
\newblock Superprocesses in random environments.
\newblock {\em Ann. Probab.}, 24(4):1953--1978, 1996.
\newblock \href {https://doi.org/10.1214/aop/1041903212}
  {\path{doi:10.1214/aop/1041903212}}.

\bibitem[Pet95]{Petrov1995}
V.~V. Petrov.
\newblock {\em Limit theorems of probability theory}, volume~4 of {\em Oxford
  Studies in Probability}.
\newblock The Clarendon Press, Oxford University Press, New York, 1995.
\newblock Sequences of independent random variables, Oxford Science
  Publications.

\bibitem[Rei89]{Reimers1989}
M.~Reimers.
\newblock One-dimensional stochastic partial differential equations and the
  branching measure diffusion.
\newblock {\em Probab. Theory Related Fields}, 81(3):319--340, 1989.
\newblock \href {https://doi.org/10.1007/BF00340057}
  {\path{doi:10.1007/BF00340057}}.

\bibitem[Ros20]{Rosati2019kRSBM}
T.~C. Rosati.
\newblock Killed rough super-{B}rownian motion.
\newblock {\em Electron. Commun. Probab.}, 25:Paper No. 44, 12, 2020.
\newblock \href {https://doi.org/10.1214/20-ecp319}
  {\path{doi:10.1214/20-ecp319}}.

\bibitem[Wal86]{Walsh1986}
J.~B. Walsh.
\newblock An introduction to stochastic partial differential equations.
\newblock In P.~L. Hennequin, editor, {\em {\'E}cole d'{\'E}t{\'e} de
  Probabilit{\'e}s de Saint Flour XIV - 1984}, pages 265--439, Berlin,
  Heidelberg, 1986. Springer Berlin Heidelberg.

\bibitem[ZMRS87]{ZeldocivhMolchanovRuzmauikinSokolov1987}
Y.~B. Zel'dovich, S.~A. Molchanov, A.~A. Ruzma{\u{\i}}kin, and D.~D. Sokolov.
\newblock Intermittency in random media.
\newblock {\em Uspekhi Fiz. Nauk}, 152(1):3--32, 1987.
\newblock \href {https://doi.org/10.1070/PU1987v030n05ABEH002867}
  {\path{doi:10.1070/PU1987v030n05ABEH002867}}.

\end{thebibliography}

\end{document}